\pdfoutput=1
\RequirePackage{ifpdf}
\ifpdf 
\documentclass[pdftex]{sigma}
\else
\documentclass{sigma}
\fi

\usepackage[all]{xy}
\usepackage{cancel}

\numberwithin{equation}{section}

\newtheorem{Theorem}{Theorem}[section]
\newtheorem{lem}[Theorem]{Lemma}
\newtheorem{prop}[Theorem]{Proposition}
\newtheorem{Corollary}[Theorem]{Corollary}
{ \theoremstyle{definition}
\newtheorem{Definition}[Theorem]{Definition}
\newtheorem{Remark}[Theorem]{Remark} }

\DeclareMathOperator{\pa}{p-area}

\begin{document}


\newcommand{\arXivNumber}{1509.00950}

\renewcommand{\PaperNumber}{097}

\FirstPageHeading

\ShortArticleName{An Application of the Moving Frame Method}

\ArticleName{An Application of the Moving Frame Method\\ to Integral Geometry in the Heisenberg Group}

\Author{Hung-Lin CHIU~$^\dag$, Yen-Chang HUANG~$^\ddag$ and Sin-Hua LAI~$^\dag$}

\AuthorNameForHeading{H.-L.~Chiu, Y.-C.~Huang and S.-H.~Lai}

\Address{$^\dag$~Department of Mathematics, National Central University, Chung Li, Taiwan}
\EmailD{\href{mailto:hlchiu@math.ncu.edu.tw}{hlchiu@math.ncu.edu.tw}, \href{mailto:972401001@cc.ncu.edu.tw}{972401001@cc.ncu.edu.tw}}

\Address{$^\ddag$~School of Mathematics and Statistics, Xinyang Normal University, Henan, P.R.~China}
\EmailD{\href{mailto:ychuang@xynu.edu.cn}{ychuang@xynu.edu.cn}}

\ArticleDates{Received March 09, 2017, in f\/inal form December 09, 2017; Published online December 26, 2017}

\Abstract{We show the fundamental theorems of curves and surfaces in the 3-dimensional Heisenberg group and f\/ind a complete set of invariants for curves and surfaces respectively. The proofs are based on Cartan's method of moving frames and Lie group theory. As an application of the main theorems, a Crofton-type formula is proved in terms of p-area which naturally arises from the variation of volume. The application makes a connection between CR geometry and integral geometry.}


\Classification{53C15; 53C65; 32V20}

\section{Introduction}
In Euclidean spaces, the fundamental theorem of curves states that any unit-speed curve is completely determined by its curvature and torsion. More precisely, given two functions $k(s)$ and $\tau (s)$ with $k(s)>0$, there exists a unit-speed curve whose curvature and torsion
are the functions $k$ and $\tau $, respectively, uniquely up to a Euclidean rigid motion. We present the analogous theorems of curves and surfaces in the $3$-dimensional Heisenberg group $H_1$. The structure of the group of transformations in $H_1$, which is similar to the group of rigid motions in Euclidean spaces, is also studied. Moreover, we develop the concept of the geometric invariants for curves and surfaces in the above sense. It should be emphasised that owning such invariants helps us to understand the geometric structures in CR manifolds and to develop the applications to integral geometry.

We give a brief review of the Heisenberg group. All the details can be found in \cite{CCG, CHMY2}. The Heisenberg group~$H_1$ is the space $\mathbb{R}^{3}$ associated with the group multiplication
\begin{gather*}
(x_{1},y_{1},z_{1})\circ (x_{2},y_{2},z_{2})=(x_{1}+x_{2},y_{1}+y_{2},z_{1}+z_{2}+y_{1}x_{2}-x_{1}y_{2}),
\end{gather*}
which is also a $3$-dimensional Lie group. The standard left-invariant vector f\/ields in $H_1$
\begin{gather*}
\mathring{e}_{1}=\frac{\partial }{\partial x}+y\frac{\partial }{\partial z},\qquad \mathring{e}_{2}=\frac{\partial }{\partial y}-x\frac{\partial }{\partial z},\qquad \text{and}\qquad T=\frac{\partial }{\partial z}
\end{gather*}
form a basis of the vector space of left-invariant vector f\/ields, where $(\frac{\partial}{\partial x},\frac{\partial}{\partial y},\frac{\partial}{\partial z})$ denotes the standard basis in~$\mathbb{R}^3$. The standard contact bundle $\xi:=\operatorname{span}\{\mathring{e}_1,\mathring{e}_2\}$ in $H_1$ is a subbundle of the tangent bundle $TH_1$. Equivalently, the contact bundle can be def\/ined as
\begin{gather*}
\xi=\ker\Theta,
\end{gather*}
where
\begin{gather*}
\Theta={\rm d}z+x{\rm d}y-y{\rm d}x
\end{gather*}
is the standard contact form. $T$ is called the Reeb vector f\/ield and $\Theta(T)=1$.
A CR structure on $H_1$ is an endomorphism $J\colon \xi\rightarrow \xi$ def\/ined by
\begin{gather*}
J(\mathring{e}_{1})=\mathring{e}_{2}\qquad \text{and}\qquad J(\mathring{e}_{2})=-\mathring{e}_{1}.
\end{gather*}
For any vectors $X,Y\in \xi$, we can associate a natural metric
\begin{gather*}
h(X, Y):={\rm d}\Theta(X, JY)
\end{gather*}called Levi-metric~\cite[Section~2]{CHMY2}. The metric $g_\Theta:=h \oplus \Theta^2$ is the adapted metric def\/ined on the tangent bundle $TH_1$.

The Heisenberg group $H_1$ can be regarded as a pseudo-hermitian manifold by con\-sidering $H_1$ associated with the standard pseudo-hermitian structure $(J,\Theta)$. Recall that a pseudo-hermitian transformation on $H_1$ is a dif\/feomorphism on $H_1$ preserving the pseudo-hermitian structure~$(J,\Theta)$. For more information about pseudo-hermitian structure, we refer the readers to \cite{CHMY1, Le1, Le2, We}. Denote by ${\rm PSH}(1)$ the group of pseudo-hermitian transformations on $H_1$, and call any element in ${\rm PSH}(1)$ a \textit{symmetry}. A symmetry in $H_1$ plays the same role as a rigid motion in $\mathbb{R}^n$ and will be characterized in Section~\ref{pseutran}.

Let $\gamma\colon I\rightarrow H_1$ be a parametrized curve. For any $t\in I$, the velocity $\gamma^{\prime}(t)$ has the natural decomposition
\begin{gather*}
\gamma^{\prime}(t)=\gamma^{\prime}_{\xi}(t)+\gamma^{\prime}_{T}(t),
\end{gather*}
where $\gamma^{\prime}_{\xi}(t)$ and $\gamma^{\prime}_{T}(t)$ are, respectively, the orthogonal projection of~$\gamma^{\prime}(t)$ on $\xi$ along $T$ and the orthogonal projection of $\gamma^{\prime}(t)$ on $T$ along $\xi$ with respect to the adapted metric~$g_\Theta$.

\begin{Definition}\label{definition2}
A \textit{horizontally regular curve} is a parametrized curve $\gamma(t)$
such that $\gamma^{\prime}_{\xi}(t)\neq 0$ for all $t\in I $. We say that $\gamma (t)$ is a \textit{horizontal curve} if $\gamma ^{\prime}(t)=\gamma
_{\xi }^{\prime}(t)$ for all $t\in I $.
\end{Definition}

In the context of contact geometry, some authors call the horizontally regular curves \textit{Le\-gendrian curves}, for example, in \cite{FT,Gi, G, MM}. Proposition~\ref{norpara} shows that any horizontally regular curve can be uniquely reparametrized by \textit{horizontal arc-length}~$s$, up to a~constant, such that $|\gamma _{\xi}^{\prime}(s)|=1$ for all $s$, and called the curve being \textit{with horizontal unit-speed}. Throughout the article, we always take for granted that the length $|\cdot |$ and the inner product $\langle\cdot,\cdot\rangle$ are def\/ined on the contact bundle in the sense of Levi-metric.

For a horizontally regular curve $\gamma(s)$ parametrized by horizontal arc-length $s$, we def\/ine the \textit{$p$-curvature} $k(s)$ and the \textit{contact normality} $\tau(s)$ by
\begin{gather*}
k(s):=\left\langle\frac{{\rm d}X(s)}{{\rm d}s},Y(s)\right\rangle, \qquad \tau(s):=\langle\gamma^{\prime}(s),T\rangle,
\end{gather*}
where $X(s)=\gamma^{\prime}_{\xi}(s)$, $Y(s)=JX(s)$, and $\frac{{\rm d}X(s)}{{\rm d}s}$ denotes the derivative of $X(s)$ w.r.t.\ the arc-length~$s$. Note that $k(s)$ is analogous to the curvature of the curve in $\mathbb{R}^n$, while $\tau(s)$ measures how far the curve is from being horizontal. We also point out that~$k(s)$ and~$\tau(s)$ are invariant under pseudo-hermitian transformations of horizontally regular curves. Recently we generalize those invariants to the higher dimension $H_n$ for any $n\geq 1$ and study the problem of classif\/ication of horizontal curves~\cite{CFH}.

The f\/irst theorem shows that horizontally regular curves are completely characterized by the functions $k(s)$ and $\tau(s)$.
\begin{Theorem}[the fundamental theorem for curves in $H_1$]\label{main1}
Given $C^1$-functions $k(s)$, $\tau(s)$, there exists a horizontally regular curve $\gamma(s)$ with horizontal unit-speed having $k(s)$ and $\tau(s)$ as its $p$-curvature and contact normality, respectively. In addition, any regular curve $\widetilde\gamma(s)$ with horizontal unit-speed satisfying the same $p$-curvature $k(s)$ and contact normality $\tau(s)$ differs from $\gamma(s)$ by a pseudo-hermitian transformation $g\in {\rm PSH}(1)$, namely,
\begin{gather*}
\widetilde\gamma{(s)}=g\circ \gamma(s)
\end{gather*}
for all $s$.
\end{Theorem}

Since a curve $\gamma (s)$ is horizontal if and only if the contact normality $\tau(s)=0$, we immediately have the corollary.

\begin{Corollary}
Given a $C^1$-function $k(s)$, there exists a horizontal curve $\gamma(s)$ with horizontal unit-speed having $k(s)$ as its $p$-curvature. In addition, any horizontal curve $\widetilde\gamma(s)$ with horizontal unit-speed satisfying the same $p$-curvature differs from $\gamma(s)$ by a pseudo-hermitian transformation $g\in {\rm PSH}(1)$, namely,
\begin{gather*}
\widetilde\gamma{(s)}=g\circ \gamma(s)
\end{gather*}
for all $s$.
\end{Corollary}

If the horizontally regular curve $\gamma$ is not parametrized by horizontal arc-length, in Section~\ref{computofpt} we also obtain the explicit formulae for the $p$-curvature and the contact normality.
\begin{Theorem}\label{main2}
Let $\gamma(t)= (x(t),y(t),z(t) )\in H_1$ be a~horizontally regular curve, not necessarily with horizontal unit-speed. The $p$-curvature $k(t)$ and the contact normality $\tau(t)$ of $\gamma(s)$ are
\begin{gather} \label{curformula}
k(t)=\frac{x^{\prime}y^{\prime\prime}-x^{\prime \prime}y^{\prime}}{\big((x^{\prime})^{2}+(y^{\prime})^{2}\big)^{\frac{
3}{2}}}(t), \qquad
\tau(t)=\frac{xy^{\prime}-x^{\prime}y+z^{\prime}}{\big((x^{\prime})^{2}+(y^{\prime})^{2}\big)^{\frac{1}{2}}}(t).
\end{gather}
\end{Theorem}

Notice that in \eqref{curformula} the $p$-curvature $k(t)$ depends only on $x(t)$, $y(t)$. We observe that $k(t)$ is the signed curvature of the plane curve $\alpha(t):=\pi \circ \gamma(t)=\big(x(t),y(t)\big)$, where $\pi$ is the projection onto the $xy$-plane along the $z$-axis. It is the fact that the signed curvature of a given plane curve completely describes the curve's behavior, we have the corollary:

\begin{Corollary}Suppose two horizontally regular curves in $H_1$ differ by a Heisenberg rigid motion, then their projections onto the $xy$-plane along the $z$-axis differ by a Euclidean rigid motion. In particular, two horizontal curves in~$H_1$ differ by a Heisenberg rigid motion if and only if their projections are congruent in the Euclidean plane.
\end{Corollary}

As an example, we calculate the $p$-curvature and contact normality for the geodesics, and obtain the
characteristic description of the geodesics.

\begin{Theorem}\label{chaofgeo} In $H_1$, the geodesics are the horizontally regular curves with constant $p$-curvature and zero contact normality.
\end{Theorem}

The second part of the paper shows the fundamental theorem of surfaces in $H_1$. Although the theorem has been generalized to hypersurfaces embedded in $H_n$ for any $n\geq 1$ (see~\cite{CL}), for the sake of being self-contained and future studies, we give a simpler proof for the case $n=1$ (Theorem~\ref{main8}). It is also worth to mention that Def\/inition~\ref{definition1} and the proofs of Theorems~\ref{main4} and~\ref{main8} are more primitive but intuitive than the one in \cite[Theorem~1.7]{CL}.

Let $\Sigma\subset H_1$ be an embedded regular surface. Recall that a singular point $p\in \Sigma $ is a point such that the tangent plane $T_{p}\Sigma $ coincides with the contact plane $\xi_{p}$ at $p$. Therefore outside the singular set (the non-singular part of $\Sigma $), the line bundle $T\Sigma \cap \xi$ forms one-dimensional foliation, which is called \textit{characteristic foliation}.

\begin{Definition}\label{definition1}
Let $F\colon U\rightarrow H_1$ be a parametrized surface with coordinates $(u,v)$ on $U\subset \mathbb{R}^2$. We say that $F$ is a \textit{normal parametrization} if
\begin{enumerate}\itemsep=0pt
\item[1)] $F(U)$ is a surface without singular points,

\item[2)] $F_{u}:=\frac{\partial F}{\partial u}$ def\/ines the characteristic foliation on $F(U)$,

\item[3)] $|F_{u}|=1$ for each point $(u,v)\in U$, where the norm is with respect to the Levi-metric.
\end{enumerate}
We call $(u,v)$ \textit{normal coordinates} of the surface $F(U)$.
\end{Definition}

It is easy to see that normal coordinates always exist locally near a non-singular point $p\in \Sigma$. In addition, for a normal parametrization $F$, denote $X=F_{u}$, $Y=JX$ and $T=\frac{\partial}{\partial z}$, we def\/ine the smooth functions $a$, $b$, $c$, $l$ and~$m$ on~$U$ by
\begin{gather}\label{coeofform}
a:=\langle F_{v},X\rangle ,\! \qquad b:=\langle F_{v},Y\rangle , \!\qquad c=:\langle F_{v},T\rangle , \!\qquad
l:=\langle F_{uu},Y\rangle,\! \qquad m:=\langle F_{uv},Y\rangle , \!\!\!
\end{gather}and call $a$, $b$ and $c$ \textit{the coefficients of the
first kind} of $F$, and $l$, $m$ \textit{the coefficients of the second kind}. In Section~\ref{invpasur} we calculate the Darboux derivatives (see Section~\ref{section2} and \eqref{Darder9}) of ${\rm PSH}(1)$ and it is known that by the method of moving frame, the inf\/initesimal displacement ${\rm d}X$ on the surface can be represented in terms of ${\rm d}u$ and ${\rm d}v$,
\begin{gather*}
{\rm d}X=F_{uu}{\rm d}u+F_{uv}{\rm d}v,
\end{gather*}
and so the functions $\ell$ and $m$ are the coef\/f\/icients of Darboux derivatives in terms of~${\rm d}u$ and~${\rm d}v$ respectively \eqref{Darder2}. By comparing~\eqref{Darder10} and~\eqref{Darder11}, all coef\/f\/icients satisfy the integrability conditions
\begin{gather} \label{intcons1}
a_{u}=bl,\qquad b_{u}=-al+m,\qquad c_{u}=2b, \qquad l_{v}-m_{u}=0,
\end{gather}where the subscripts denote the partial derivatives.

The following theorem states that these coef\/f\/icients are the complete dif\/ferential invariants for the map $F$.

\begin{Theorem}\label{main4} Let $U\subset \mathbb{R}^{2}$ be a simply connected open set. Suppose that $a$, $b$, $c$, $l$ and $m$ are functions defined on $U$ satisfying the integrability conditions~\eqref{intcons1}. Then there exists a normal parametrization $F\colon U\rightarrow H_1$ having $a$, $b$, $c$ and $l$, $m$ as the coefficients of first kind and second kind of~$F$, respectively. In addition, any normal parametrization $\widetilde{F}\colon U\rightarrow H_1$ with the same coefficients of first kind and second kind differ from~$F$ by a Heisenberg rigid motion, namely, $\widetilde{F}(u,v)=g\circ F(u,v)$ for all $(u,v)\in U$ for some $g\in {\rm PSH}(1)$.
\end{Theorem}

We should point out that the regularity of $F$ is, at least, $C^2$. In \eqref{tlofc2}, we will show that the function $l$, up to a sign, is independent of the choice of normal coordinates, and hence it is a~dif\/ferential invariant of the surface $F(U)$. Actually $l$ is the \textit{$p$-mean curvature} for $H_n$, $n\geq 2$ (see~\cite{CHMY2}). In particular, $F(U)$ is a $p$-minimal surface when $l=0$; such a parametrization $F\colon U\rightarrow H_1$ is called a~\textit{normal parametrization of $p$-minimal surface}. In this case, the integrability~condition \eqref{intcons1} becomes
\begin{gather}
a_{u}=0,\qquad b_{uu}=0,\qquad c_{u}=2b, \label{intcons3}\\
b_u=m, \label{secbyfir1}
\end{gather}
and the coef\/f\/icients of f\/irst kind completely dominate those of second kind. We conclude all above as the following result.

\begin{Theorem}\label{main5} Let $U\subset \mathbb{R}^{2}$ be a simply connected open set. Suppose that $a$, $b$ and $c$ are smooth functions defined on $U$ satisfying the integrability conditions \eqref{intcons3}. Then there exists a normal parametrization of $p$-minimal surface $F\colon U\rightarrow H_1$ having $a$, $b$ and $c$ as the coefficients of first kind of $F$, which also determines the coefficient $m$ of the second kind as in~\eqref{secbyfir1}. In addition, any normal parametrization of $p$-minimal surface $\widetilde{F}\colon U\rightarrow H_1$ with the same conditions differs from $F$ by a Heisenberg rigid motion, namely, $\widetilde{F}(u,v)=g\circ F(u,v)$ in $U$ for some $g\in {\rm PSH}(1)$.
\end{Theorem}

In Section~\ref{invpasur}, other invariants on the surface $\Sigma$ will also be obtained, including
\begin{gather*}\alpha :=\frac{b}{c}
\end{gather*}(up to a sign, called the \textit{$p$-variation}), and the restricted adapted metric $g_{\Theta}|_\Sigma$ on the surface $\Sigma$. Actually $\alpha$ is the function such that the vector f\/ield $\alpha e_{2}+T$ is tangent to the surface, where $e_{2}=Je_{1}$ and $e_{1}$ is a unit vector f\/ield tangent to the characteristic foliation. Let
\begin{gather*}
e_{\Sigma }:=\frac{\alpha e_{2}+T}{\sqrt{1+\alpha ^{2}}},
\end{gather*} be a unit vector f\/ield tangent to the surface.
Then we observe that these invariants $\alpha$, $l$, $e_\Sigma$ satisfy the integrability condition:
\begin{gather}
\big(1+\alpha ^{2}\big)^{\frac{3}{2}}(e_{\Sigma }l) =\big(1+\alpha
^{2}\big)(e_{1}e_{1}\alpha )-\alpha (e_{1}\alpha )^{2}+4\alpha \big(1+\alpha
^{2}\big)(e_{1}\alpha ) \nonumber\\
\hphantom{\big(1+\alpha ^{2}\big)^{\frac{3}{2}}(e_{\Sigma }l) =}{}+\alpha \big(1+\alpha ^{2}\big)^{2}K+\alpha l\big(1+\alpha ^{2}\big)^{\frac{1}{2}}(e_{\Sigma }\alpha )+\alpha \big(1+\alpha ^{2}\big)l^{2},\label{Intconsur}
\end{gather}
where $K$ is the Gaussian curvature with respect to $g_{\Theta}|_{\Sigma }$.

After studying the invariants in $H_1$, we show the second main theorem which says that the three invariants (the Riemannian metric~$g_\Theta$ induced by the adapted metric, the $p$-mean curvature~$l$, and the $p$-variation $\alpha$) comprise a complete set of invariants for a surface without singular points.

\begin{Theorem}[the fundamental theorem for surfaces in $H_1$]
\label{main8} Let $(\Sigma ,g)$ be a $2$-dimensional Riemannian manifold with Gaussian curvature $K$, and $\alpha^\prime$, $l^\prime$ two real-valued functions defined on~$\Sigma $. Assume that~$K$, $\alpha^{\prime}$ and $l^{\prime}$ satisfy the integrability condition~\eqref{Intconsur}. Then for every non-singular point $p\in \Sigma $, there exists an open neighborhood $U$ containing $p$ and an embedding $f\colon U\rightarrow H_1$ such that \begin{gather*}
g=f^{\ast }(g_{\Theta}),\qquad
\alpha ^{\prime}=f^{\ast }\alpha, \qquad
l^\prime=f^{\ast }l,
\end{gather*}
where $\alpha$, $l$ are the induced $p$-variation and $p$-mean curvature on $f(U)$ respectively. Moreover, $f$~is uniquely determined up to a~Heisenberg rigid motion.
\end{Theorem}

The third part of the paper is an application of the motion equations and the structure equations obtained from the proof of fundamental theorem for curves. We derive the Crofton formula in $H_1$ which is a classical result of integral geometry, relating the length of a f\/ixed curve, and the number of intersections for the curve and randomly oriented lines passing through it. Santal\'{o} generalized the result to compact Riemannian manifolds with boundary \cite{Ren, San}. In the simple case $\mathbb{R}^2$, given a f\/ixed piecewise regular curve $\gamma$, the Crofton formula states that
\begin{gather*}
\int_{\ell\cap\gamma\neq\varnothing}n(\ell\cap\gamma) {\rm d}L =4\cdot \operatorname{length}(\gamma),
\end{gather*}
where ${\rm d}L$ is the kinematic density def\/ined on the set of oriented lines in $\mathbb{R}^2$, and $n(\ell\cap\gamma)$ is the number of intersections of the line $\ell$ with $\gamma$. We have the analogues formula in~$H_1$. A~signif\/icant observation is that the geometric quantity on the right-hand side of the formula~\eqref{croftonidentity} is the $p$-area which naturally arises from the variation of volume for domains in CR manifolds~\cite{CHMY2}. By the similar technique, one of the authors also show the containment problem for the geometric probability in~$H_1$~\cite{H}. Recently Prandi, Rizzi, Seri~\cite{PRS} show a sub-Riemannian version of the classical Santal\'{o} formula which is applied to f\/inding the lower bound of sub-Laplacian in a~compact domain with boundary. The other approaches can be referred to~\cite{Pansu} (sub-Riemannian), and~\cite{Mon} (Carnot groups).

\begin{Theorem}[Crofton formula in $H_1$]\label{crofton}
Suppose $\mathbb{X}\colon (u,v)\in\Omega \mapsto \Sigma\subset H_1$ is a $C^2$-surface for some domain $\Omega \subset \mathbb{R}^2$. Let $\mathcal{L}$ be the set of oriented horizontal lines in $H_1$ and $n(\ell\cap\Sigma)$ be the number of intersections of the horizontal line $\ell\in\mathcal{L}$ with the surface $\Sigma$.
Then we have the Crofton formula
\begin{gather}\label{croftonidentity}
\int_{\ell\in\mathcal{L},\ \ell\cap\Sigma\neq\varnothing}n(\ell\cap\Sigma){\rm d}L=4\cdot\pa (\Sigma),
\end{gather}
where ${\rm d}L:= {\rm d}p\wedge {\rm d}\theta \wedge {\rm d}t$ is the kinematic density on $\mathcal{L}$.
\end{Theorem}

We give the outline of the paper. In Section~\ref{section2}, we state two
propositions about existence and uniqueness of mappings from a smooth manifold
into a Lie group $G$, which underlies our main theorems. In Section~\ref{section3}, we not only express the representation of $\ {\rm PSH}(1)$ but discuss how the matrix Lie group ${\rm PSH}(1)$ can be interpreted as the set of moving frames on the homogeneous space $H_1={\rm PSH}(1)/{\rm SO}(2)$; the moving frame formula in $H_1$ via the (left-invariant) Maurer--Cartan form will be derived. In Section~\ref{section4}, we compute the Darboux derivatives of the lift of a horizontally regular curve and give the proof of the f\/irst main theorem; moreover, the $p$-curvature and the contact normality for horizontally regular curves and geodesics are calculated.
In Section~\ref{invpasur}, we compute the Darboux derivatives of the lift of normal parametrized surfaces, and achieve the complete set of dif\/ferential invariants for a normal parametrized surface. In Section~\ref{section6}, by calculating the Darboux derivatives of the lift for $f\colon \Sigma\rightarrow H_1$, we show the fundamental theorem for surfaces $\Sigma$ in $H_1$. In Section~\ref{section7}, we show the Crofton formula which connects CR geometry and integral geometry.

\section{Calculus on Lie groups}\label{section2}
We recall two basic theorems from Lie groups, which play the essential roles in the proof of the main theorems. For the details we refer the readers to \cite{CC, C, G, IL,S}.

Let $M$ be a connected smooth manifold and $G\subset {\rm GL}(n,{\mathbb R})$ a matrix subgroup with Lie algebra $\mathfrak{g}$. Recall that a (left-invariant) Maurer--Cartan form $\omega$ is a Lie algebra-valued 1-form globally def\/ined on $G$ which is a linear mapping of the tangent space $T_gG$ at each $g\in G$ into~$T_eG$
\begin{gather*}
\omega(v):=\big(L_{g^{-1}}\big)_{*}v, \qquad \text{for all} \quad v\in T_g G,
\end{gather*} where $e\in G$ is the identity element. In particular when $G$ is a matrix Lie group, one has
\begin{gather*}
\omega = g^{-1}dg.
\end{gather*}

We f\/irst introduce the theorem of uniqueness.

\begin{Theorem}\label{ft1}
Given two maps $f, \widetilde{f}\colon M\rightarrow G$, then $ \widetilde{f}^{*}\omega=f^{*}\omega$ if and only if $\widetilde{f}=g\cdot f$
for some $g\in G$.
\end{Theorem}
We call the pullback 1-form $f^{*}\omega$ \textit{the Darboux derivative} of the map $f\colon M\rightarrow G$. When $M=\mathbb{R}$ and $G=(\mathbb{R}, +)$, the addition group, Theorem \ref{ft1} can be rephrased as the Fundamental Theorem of Calculus: if two dif\/ferentiable functions $f(x)$ and $\widetilde{f}(x)$ have the same derivatives $\frac{{\rm d}f}{{\rm d}x}=\frac{{\rm d}\widetilde{f}}{{\rm d}x}$, then $f(x)=\widetilde{f}(x)+c$ for some constant~$c$,

The second result is the theorem of existence.
\begin{Theorem}\label{ft2} Suppose that $\phi$ is a $\mathfrak{g}$-valued $1$-form on a~simply connected manifold $M$. Then there exists a map $f\colon M\rightarrow G$ satisfying $f^{*}\omega=\phi$ if and only if ${\rm d}\phi=-\phi \wedge \phi$. Moreover, the resulting map $f$ is unique up to a group action.
\end{Theorem}
In the case $M=\mathbb{R}$ and $G=(\mathbb{R}, +)$, this theorem implies the existence of derivatives for dif\/ferentiable functions. We mention that the proof of Theorem~\ref{ft2} relies on the Frobenius theorem.

\section[The group of pseudo-hermitian transformations on $H_1$]{The group of pseudo-hermitian transformations on $\boldsymbol{H_1}$}\label{section3}

\subsection[The pseudo-hermitian transformations on $H_1$]{The pseudo-hermitian transformations on $\boldsymbol{H_1}$}\label{pseutran}
A pseudo-hermitian transformation on $H_1$ is a dif\/feomorphism $\Phi$ on $H_1$ preserving the CR structure $J$ and the contact form $\Theta$; it satisf\/ies
\begin{gather*}
\Phi _{\ast }J=J\Phi _{\ast }\qquad \text{on}\quad \xi \qquad \text{and}\qquad
\ \Phi ^{\ast }\Theta =\Theta \qquad \text{in} \quad H_1.
\end{gather*}
A trivial example of a pseudo-hermitian transformation is a left translation $L_p$ in $H_1$; the other example is def\/ined by $\Phi _{R}\colon H_1\rightarrow H_1$
\begin{gather*}
\Phi _{R}\left(
\begin{matrix}
x \\
y \\
z
\end{matrix}
\right) \longrightarrow \left(
\begin{matrix}
R & 0 \\
0 & 1%
\end{matrix}
\right) \left(
\begin{matrix}
x \\
y \\
z
\end{matrix}
\right) ,
\end{gather*}
where $R\in {\rm SO}(2)$ is a 2$\times$2 special orthogonal matrix.

Let ${\rm PSH}(1)$ be the group of pseudo-hermitian transformations on $H_1$. We shall show that the group ${\rm PSH}(1)$ exactly consists of all
the transformations of the forms $\Phi _{p,R}:= L_{p}\circ \Phi _{R}$, a~transformation $\Phi _{R}$ followed by a left translation $L_{p}$. More precisely, we have
\begin{gather*}
\Phi _{p,R}\left(
\begin{matrix}
x \\
y \\
z%
\end{matrix}
\right) =\left(
\begin{matrix}
ax-by+p_{1} \\
bx+ay+p_{2} \\
(ap_{2}-bp_{1})x+(-bp_{2}-ap_{1})y+z+p_{3}
\end{matrix}
\right) ,
\end{gather*}
where $p=(p_{1},p_{2},p_{3})^{\rm t}\in H_{1}$ and $R=\left(
\begin{smallmatrix}
a & -b \\
b & a
\end{smallmatrix}
\right) \in {\rm SO}(2)$.

\begin{Theorem}\label{psgrp} Let $\Phi\colon H_{1}\rightarrow H_{1}$ be a pseudo-hermitian transformation. Then $\Phi=L_{p}\circ \Phi_{R}$ for some $R\in {\rm SO}(2)$ and $p\in H_{1}$.
\end{Theorem}

\begin{proof}
It suf\/f\/ices to consider the pseudo-hermitian transformation $\Phi\colon H_{1}\rightarrow H_{1}$ such that $\Phi(0)=0$. Indeed, if
$\Phi(0)=p$ for some $p\in H_1\setminus\{0\}$, then the composition $L_{p^{-1}}\circ \Phi$ is a~transformation f\/ixing the origin. Therefore, we reduce the proof of Theorem~\ref{psgrp} to the following lemma:

\begin{lem}\label{balemma} Let $\Phi $ be a pseudo-hermitian transformation on $H_{1}$
such that $\Phi (0)=0$. Then, for any $p\in H_{1}$, the matrix
representation of $\Phi _{\ast }(p)$ with respect to the standard basis $\big(\frac{\partial }{\partial x},\frac{\partial }{\partial y},\frac{\partial }{\partial z}\big)$ of~$\mathbb{R}^3$ is
\begin{gather*}
\Phi _{\ast }(p)=
\begin{pmatrix}
\cos \alpha _{0} & -\sin \alpha _{0} & 0 \\
\sin \alpha _{0} & \cos \alpha _{0} & 0 \\
0 & 0 & 1
\end{pmatrix}
_{\big( \frac{\partial }{\partial x},\frac{\partial }{\partial y},\frac{%
\partial }{\partial z}\big) }, 
\end{gather*}
for some real constant $\alpha _{0}$ which is independent of $p$, and hence $\Phi _{\ast }$ is a constant matrix.
\end{lem}

To prove Lemma~\ref{balemma}, f\/irst we calculate the matrix representation of $\Phi _{\ast }(p)$ with respect to the basis $(\mathring{e}_{1},\mathring{e}_{2},T)$. Since
\begin{gather*}
\Theta\left( \Phi _{\ast }\mathring{e}_{i}\right) =\left( \Phi ^{\ast
}\Theta\right) \left( \mathring{e}_{i}\right) =\Theta \left(
\mathring{e}_{i}\right) =0, \qquad \text{for} \quad i=1,2,
\end{gather*}
the contact bundle $\xi $ is invariant under $\Phi _{\ast }$. In addition, let $%
h$ be the Levi-metric on $\xi $ def\/ined by $h(X,Y)={\rm d}\Theta(X,JY) $, then
\begin{gather*}
\Phi ^{\ast }h(X,Y) =h(\Phi _{\ast }X,\Phi _{\ast }Y)={\rm d}\Theta(\Phi_{\ast }X,J\Phi _{\ast }Y)={\rm d}\Theta (\Phi _{\ast }X,\Phi _{\ast }JY)\\
\hphantom{\Phi ^{\ast }h(X,Y)}{} =\Phi ^{\ast }({\rm d}\Theta)(X,JY)={\rm d}(\Phi ^{\ast }\Theta )(X,JY) ={\rm d}\Theta(X,JY)=h(X,Y),
\end{gather*}
and hence $h ( \Phi _{\ast }X,\Phi _{\ast }Y) =h( X,Y) $ for every $X,Y\in \xi$. Thus, $\Phi _{\ast }$ is
orthogonal on~$\xi$. On the other hand, since
\begin{gather*}
\Theta (\Phi _{\ast }T)=\Theta\left( \Phi _{\ast }\frac{\partial }{\partial z}\right) = \Phi ^{\ast }\Theta \left( \frac{\partial }{\partial z}\right) =\Theta\left( \frac{\partial }{\partial z}\right) =1,
\end{gather*}
and
\begin{gather*}
{\rm d}\Theta(X,\Phi _{\ast }T) ={\rm d}\Theta\big (\Phi _{\ast }\Phi _{\ast}^{-1}X,\Phi _{\ast }T\big)=(\Phi ^{\ast }{\rm d}\Theta)\big(\Phi _{\ast }^{-1}X,T\big)=({\rm d}\Phi ^{\ast }\Theta)\big(\Phi _{\ast }^{-1}X,T\big)\\
\hphantom{{\rm d}\Theta(X,\Phi _{\ast }T)}{} ={\rm d}\Theta\big(\Phi_{\ast }^{-1}X,T\big)=0
\end{gather*} for all $X\in \xi$,
we have $\Phi _{\ast}T=T$. From the above argument, we conclude that the matrix representation
\begin{gather*}
\Phi _{\ast }(p)=
\begin{pmatrix}
\cos \alpha (p) & -\sin \alpha (p) & 0 \\
\sin \alpha (p) & \cos \alpha (p) & 0 \\
0 & 0 & 1%
\end{pmatrix}%
_{\left( \mathring{e}_{1},\mathring{e}_{2},\frac{\partial }{\partial z}
\right) },
\end{gather*}%
for some real-valued function $\alpha $ on $H_{1}$.

Next, we rewrite the matrix representation of $\Phi _{\ast }(p)$ from the basis $\big( \mathring{e}%
_{1},\mathring{e}_{2},\frac{\partial }{\partial z}\big) $ to the basis $\big( \frac{%
\partial }{\partial x},\frac{\partial }{\partial y},\frac{\partial }{%
\partial z}\big) $. Let $\Phi =(\Phi ^{1},\Phi ^{2},\Phi ^{3})$, $p= ( p_{1},p_{2},p_{3} ) $, $\mathring{e}%
_{1}(p) =\frac{\partial }{\partial x}+p_{2}\frac{\partial }{%
\partial z}$ and $\mathring{e}_{2}(p) =\frac{\partial }{\partial
y}-p_{1}\frac{\partial }{\partial z}$, then
\begin{gather*}
\Phi _{\ast }(p)\left( \frac{\partial }{\partial x}\right) =\Phi _{\ast
}(p)\left[ \mathring{e}_{1}(p) -p_{2}\frac{\partial }{\partial z}
\right] =\Phi _{\ast }(p)\left[ \mathring{e}_{1}(p) \right]
-p_{2}\frac{\partial }{\partial z} \\
\hphantom{\Phi _{\ast }(p)\left( \frac{\partial }{\partial x}\right)}{} =\cos \alpha (p) \mathring{e}_{1}\left[ \Phi (p)\right] +\sin
\alpha (p) \mathring{e}_{2}\left[ \Phi (p)\right] -p_{2}\frac{%
\partial }{\partial z} \\
\hphantom{\Phi _{\ast }(p)\left( \frac{\partial }{\partial x}\right)}{} =\cos \alpha (p) \frac{\partial }{\partial x}+\sin \alpha
(p) \frac{\partial }{\partial y} +\big[ \cos \alpha (p) \Phi ^{2}(p) -\sin
\alpha (p) \Phi ^{1}(p) -p_{2}\big] \frac{\partial
}{\partial z},
\end{gather*}
and
\begin{gather*}
\Phi _{\ast }(p)\left( \frac{\partial }{\partial y}\right) =\Phi _{\ast
}(p)\left[ \mathring{e}_{2}(p) +p_{1}\frac{\partial }{\partial z}%
\right] =\Phi _{\ast }(p)\left[ \mathring{e}_{2}(p) \right]
+p_{1}\frac{\partial }{\partial z} \\
\hphantom{\Phi _{\ast }(p)\left( \frac{\partial }{\partial y}\right)}{} =-\sin \alpha (p) \mathring{e}_{1}\left[ \Phi (p)\right] +\cos
\alpha (p) \mathring{e}_{2}\left[ \Phi (p)\right] +p_{1}\frac{%
\partial }{\partial z} \\
\hphantom{\Phi _{\ast }(p)\left( \frac{\partial }{\partial y}\right)}{} =-\sin \alpha (p) \frac{\partial }{\partial x}+\cos \alpha
(p) \frac{\partial }{\partial y} +\big[{-}\sin \alpha (p) \Phi ^{2}(p) -\cos
\alpha (p) \Phi ^{1}(p) +p_{1}\big] \frac{\partial
}{\partial z}.
\end{gather*}
Thus,
\begin{gather}\label{31}
\Phi _{\ast }(p)=
\begin{pmatrix}
\cos \alpha (p) & -\sin \alpha (p) & 0 \\
\sin \alpha (p) & \cos \alpha (p) & 0 \\
\Phi _{x}^{3}(p) & \Phi _{y}^{3}(p) & 1%
\end{pmatrix}%
_{\big( \frac{\partial }{\partial x},\frac{\partial }{\partial y},\frac{%
\partial }{\partial z}\big)}
:=
\begin{pmatrix}
\Phi^1_x & \Phi^1_y & \Phi^1_z \\
\Phi^2_x & \Phi^2_y & \Phi^2_z \\
\Phi^3_x & \Phi^3_y & \Phi^3_z \\
\end{pmatrix},
\end{gather}
where
\begin{gather*}
\Phi _{x}^{3}(p) :=\frac{\partial\Phi_x^3}{\partial x}=\cos \alpha (p) \Phi ^{2}(p)-\sin \alpha (p) \Phi ^{1}(p) -p_{2}, \\
\Phi _{y}^{3}(p):=\frac{\partial\Phi_y^3}{\partial y} =-\sin \alpha (p) \Phi ^{2}(p)-\cos \alpha (p) \Phi ^{1}(p) +p_{1},
\end{gather*}
and denote the subscripts as the partial derivatives for all $\Phi^i$'s. By \eqref{31} that $\Phi _{z}^{1}=\Phi _{z}^{2}=0$, it follows that the functions $\Phi ^{1}$ and $\Phi ^{2}$ both depend only on $x$ and $y$, and so is $\alpha $. Moreover, use~\eqref{31} again and the facts $\Phi _{xy}^{1}=\Phi _{yx}^{1}$ and $\Phi _{xy}^{2}=\Phi _{yx}^{2}$, we have
\begin{gather*}
\begin{pmatrix}
\cos \alpha & -\sin \alpha \\
\sin \alpha & \cos \alpha
\end{pmatrix}
\begin{pmatrix}
\alpha _{x} \\
\alpha _{y}%
\end{pmatrix}%
=
\begin{pmatrix}
0 \\
0
\end{pmatrix},
\end{gather*}
which implies that $\alpha _{x}=\alpha _{y}=0$. Thus $\alpha $ is a constant on $H_{1}$, say $\alpha =\alpha _{0}$. From~\eqref{31} and notice that $\Phi (0)=0$, we f\/inally get
\begin{gather*}
\Phi ^{1} =x\cos {\alpha _{0}}-y\sin {\alpha _{0}}, \qquad
\Phi ^{2} =x\sin {\alpha _{0}}+y\cos {\alpha _{0}},
\end{gather*}
which implies that $\Phi _{x}^{3}=\Phi _{y}^{3}=0$. Therefore
\begin{gather*}
\Phi _{\ast }(p)=
\begin{pmatrix}
\cos \alpha _{0} & -\sin \alpha _{0} & 0 \\
\sin \alpha _{0} & \cos \alpha _{0} & 0 \\
0 & 0 & 1%
\end{pmatrix}
_{\big( \frac{\partial }{\partial x},\frac{\partial }{\partial y},\frac{%
\partial }{\partial z}\big) },
\end{gather*}
and the result follows.
\end{proof}

\subsection[Representation of ${\rm PSH}(1)$]{Representation of $\boldsymbol{{\rm PSH}(1)}$}

The pseudo-hermitian transformation $\Phi_{p,R}$ and the points $(x,y,z)^t$ in $H_{1}$ can be respectively represented as
\begin{gather*}
\Phi_{p,R}\leftrightarrow M=\left(\begin{matrix}
1 & 0 & 0 & 0 \\
p_{1} & a & -b & 0 \\
p_{2} & b & a & 0 \\
p_{3} & ap_{2}-cp_{1} & bp_{2}-dp_{1} & 1%
\end{matrix}
\right),
\end{gather*}
and
\begin{gather*}
\left(
\begin{matrix}
x \\
y \\
z%
\end{matrix}
\right)\leftrightarrow X=\left(
\begin{matrix}
1 \\
x \\
y \\
z%
\end{matrix}
\right)
\end{gather*}
satisfying
\begin{gather*}
MX=\left(
\begin{matrix}
1 \\
\Phi_{p,R}\left(%
\begin{matrix}
x \\
y \\
z%
\end{matrix}
\right)
\end{matrix}
\right),
\end{gather*}where $a^2+b^2=1$.
Therefore, ${\rm PSH}(1)$ can be represented as a matrix group
\begin{gather*}
{\rm PSH}(1)=\left \{ M\in {\rm GL}(4,{\mathbb R})\ \Big{|}\ M=\left(
\begin{matrix}
1 & 0 & 0 & 0 \\
p_{1} & a & -b & 0 \\
p_{2} & b & a & 0 \\
p_{3} & ap_{2}-bp_{1} & -bp_{2}-ap_{1} & 1%
\end{matrix}%
\right), \, a^2+b^2=1 \right \}.
\end{gather*}

Let ${\mathfrak{psh}}(1)$ be the Lie algebra of ${\rm PSH}(1)$. It is easy to see that the element of ${\mathfrak{psh}}(1)$ is of the form
\begin{gather*}
\left(%
\begin{matrix}
0 & 0 & 0 & 0 \\
x_{1} & 0 & -x_{1}{}^{2} & 0 \\
x_{2} & x_{1}{}^{2} & 0 & 0 \\
x_{3} & x_{2} & -x_{1} & 0%
\end{matrix}
\right).
\end{gather*}
and the corresponding Maurer--Cartan form of ${\rm PSH}(1)$ is of the form
\begin{gather*}
\omega =\left(
\begin{matrix}
0 & 0 & 0 & 0 \\
\omega^{1} & 0 & -\omega_{1}{}^{2} & 0 \\
\omega^{2} & \omega_{1}{}^{2} & 0 & 0 \\
\omega^{3} & \omega^{2} & -\omega^{1} & 0%
\end{matrix}
\right),
\end{gather*}
where ${\omega_{1}}^{2}$ and $\omega ^{j}$, $j=1,2,3$, are 1-forms on ${\rm PSH}(1)$.

\subsection[The oriented frames on $H_1$]{The oriented frames on $\boldsymbol{H_1}$}

The oriented frame $(p;X,Y,T)$ on $H_1$ consists of the point $p\in H_1$ and the orthonormal vector f\/ields
 $X\in \xi_p$, $Y=JX$ with respect to the Levi-metric. We can identify ${\rm PSH}(1)$
with the set of all oriented frames on $H_{1}$ as follows:
\begin{gather*}
{\rm PSH}(1)\ni M=\left(
\begin{matrix}
1 & 0 & 0 & 0 \\
p_{1} & a & -b & 0 \\
p_{2} & b & a & 0 \\
p_{3} & ap_{2}-bp_{1} & -bp_{2}-ap_{1} & 1
\end{matrix}
\right) \leftrightarrow (p;X,Y,T),
\end{gather*}
where
\begin{gather*}
p =(p_{1},p_{2},p_{3})^{\rm t}, \\
X =a\frac{\partial }{\partial x}+b\frac{\partial }{\partial y}+(ap_{2}-bp_{1})\frac{\partial }{\partial t}, \qquad
Y =-b\frac{\partial }{\partial x}+a\frac{\partial }{\partial y}+(-bp_{2}-ap_{1})\frac{\partial }{\partial t}.
\end{gather*}
Actually, we have $X=a\mathring{e}_{1}(p)+b\mathring{e}_{2}(p)$ and $Y=-b\mathring{e}_{1}(p)+a\mathring{e}_{2}(p)$, and hence $M$ is the unique $4\times
4 $ matrix such that
\begin{gather*}
(p;X,Y,T)=(0;\mathring{e}_{1},\mathring{e}_{2},T)M.
\end{gather*}

\subsection{Moving frame formula}

Since ${\rm PSH}(1)$ is a matrix Lie group, the Maurer--Cartan form must be $\omega =M^{-1}{\rm d}M$ or ${\rm d}M=M\omega $ (see \cite{CCL}). Immediately one has that
\begin{gather*}
({\rm d}p;{\rm d}X,{\rm d}Y,{\rm d}T)=(p;X,Y,T)\left(
\begin{matrix}
0 & 0 & 0 & 0 \\
\omega ^{1} & 0 & -\omega _{1}{}^{2} & 0 \\
\omega ^{2} & \omega _{1}{}^{2} & 0 & 0 \\
\omega ^{3} & \omega ^{2} & -\omega ^{1} & 0%
\end{matrix}%
\right).
\end{gather*}
Thus, we have reached the moving frame formula:
\begin{gather}
{\rm d}p=\omega ^{1}X+\omega ^{2}Y+\omega ^{3}T, \qquad
{\rm d}X=\ \omega_{1}{}^{2}Y+\omega ^{2} T,\nonumber\\
{\rm d}Y= -\omega _{1}{}^{2}X-\omega ^{1}T,\qquad
{\rm d}T=0.\label{movingframe}
\end{gather}

\section[Dif\/ferential invariants of horizontally regular curves in $H_1$]{Dif\/ferential invariants of horizontally regular curves in $\boldsymbol{H_1}$}\label{section4}

\begin{prop}\label{norpara}
Any horizontally regular curve $\gamma(t)$ can be reparametrized by its horizontal arc-length $s$ such that $|\gamma'_\xi(s)|=1$.
\end{prop}

\begin{proof}
Def\/ine $s(t)=\int_{0}^{t}|\gamma _{\xi}^{\prime}(u)|{\rm d}u$. Then any horizontal arc-length dif\/fers $s$ up to a constant. By the fundamental theorem
of calculus, we have $\frac{{\rm d}s}{{\rm d}t}=|\gamma _{\xi}^{\prime}(t)|$.
Since
\begin{gather*}
\frac{{\rm d}\gamma }{{\rm d}s}=\frac{{\rm d}\gamma }{{\rm d}t}\frac{{\rm d}t}{{\rm d}s}=\frac{\gamma ^{\prime}(t)}{|\gamma _{\xi}^{\prime}(t)|},
\end{gather*}
$\gamma _{\xi}^{\prime}(s)=\frac{\gamma _{\xi}^{{\prime}}(t)}{|\gamma _{\xi}^{\prime}(t)|}$, namely, $|\gamma _{\xi
}^{\prime}(s)|=1$.
\end{proof}

\begin{Definition}
A \textit{lift} of a mapping $f\colon M\rightarrow G/H$ is def\/ined to be a map $F\colon M\rightarrow G$ such that the following diagram commutes:
\begin{gather*}
 \xymatrix{ & G \ar[d]\\ M\ar[ur]^F \ar[r]_f & G/H, }
\end{gather*}
where $G$ is a Lie group, $H$ is a closed Lie subgroup and $G/H$ is the associated homogeneous space. In additional, another lift $\tilde{F}$ of $f$ has to satisfy
\begin{gather*}
\widetilde{F}(x) = F(x)g(x)
\end{gather*}
for some map $g\colon M\rightarrow H$.
\end{Definition}
\begin{Remark}\label{settings}
In the next section, we shall set $G={\rm PSH}(1)$, $M=(a,b)\subset R$, $f=\gamma$, $F=\widetilde{\gamma}$, $G/H={\rm PSH}(1)/{\rm SO}(2)$, and identify ${\rm PSH}(1)/{\rm SO}(2)$ with $H_1$.
\end{Remark}

\subsection{The Proof of Theorem \protect \ref{main1}}

By Proposition \ref{norpara}, we may assume that the horizontally regular curve $\gamma (s)$ is parametrized by the horizontal arc-length $s$. Each point on $\gamma$ uniquely def\/ines an oriented frame
\begin{gather*}
(\gamma (s);X(s),Y(s),T),
\end{gather*}
where $X(s)=\gamma _{\xi}^{\prime}(s)$ is the horizontally tangent vector of $\gamma(s)$ and $Y(s)=JX(s)$. By Remark~\ref{settings}, there exists a lift $\tilde{\gamma}$ of $\gamma$ to ${\rm PSH}(1)$, which is unique up to a~${\rm SO}(2)$ group action. We abuse the notation and denote the lift by
\begin{gather*}
\widetilde{\gamma}(s)=(\gamma (s);X(s),Y(s),T).
\end{gather*}

Let $\omega $ be the Maurer--Cartan form of ${\rm PSH}(1)$. We shall derive the Darboux derivative $\widetilde{\gamma }^{\ast }\omega $ of the lift $\widetilde{\gamma }(s)$: by using the moving frame formula~(\ref{movingframe}), we have
\begin{gather}\label{diffcurve}
{\rm d}\widetilde{\gamma }(s)=\widetilde{\gamma }^{\ast }{\rm d}p =X(s)\widetilde{\gamma }^{\ast }\omega ^{1}+Y(s)\widetilde{\gamma }^{\ast
}\omega ^{2}+T\widetilde{\gamma }^{\ast }\omega ^{3},
\end{gather}
and observe that all pull-back 1-forms by $\widetilde{\gamma }$ are the multiples of ${\rm d}s$,
\begin{gather}\label{diffcurve1}
{\rm d}\widetilde{\gamma }(s) =\gamma _{\xi}^{\prime}(s){\rm d}s+\gamma_{T}^{\prime}(s){\rm d}s =X(s){\rm d}s+\gamma _{T}^{\prime}(s){\rm d}s.
\end{gather}
Comparing (\ref{diffcurve}) and (\ref{diffcurve1}) to get
\begin{gather*}
\widetilde{\gamma }^{\ast }\omega ^{1} ={\rm d}s,\qquad
\widetilde{\gamma }^{\ast}\omega ^{2}=0, \qquad
\widetilde{\gamma }^{\ast }\omega ^{3} =\langle \gamma ^{\prime}(s),T\rangle {\rm d}s=\tau
(s){\rm d}s.
\end{gather*}
Insert $\widetilde{\gamma }^{\ast }\omega ^{3}$ into (\ref{movingframe}),
\begin{gather*}
{\rm d}X(s)=Y(s)\widetilde{\gamma }^{\ast }\omega _{1}{}^{2}+T\widetilde{\gamma }%
^{\ast }\omega ^{2}=Y(s)\widetilde{\gamma }^{\ast }\omega _{1}{}^{2},
\end{gather*}%
one has
\begin{gather*}
\widetilde{\gamma }^{\ast }\omega _{1}{}^{2}=\left\langle \frac{{\rm d}X(s)}{{\rm d}s},Y(s)\right\rangle {\rm d}s=k(s){\rm d}s.
\end{gather*}
As a consequence, the Darboux derivative of $\widetilde{\gamma }$ is obtained
\begin{gather}\label{fodar}
\widetilde{\gamma }^{\ast }\omega =\left(
\begin{array}{cccc}
0 & 0 & 0 & 0 \\
1 & 0 & -k(s) & 0 \\
0 & k(s) & 0 & 0 \\
\tau (s) & 0 & -1 & 0%
\end{array}%
\right) {\rm d}s.
\end{gather}%

For any functions $k(s)$ and $\tau (s)$ def\/ined on an open
interval $I$. Suppose $\varphi$ is the ${\mathfrak{psh}}(1)$-valued 1-form def\/ined by (\ref{fodar}). It is easy to check that $\varphi$ satisf\/ies
${\rm d}\varphi +\varphi \wedge \varphi =0$. Therefore, Theo\-rem~\ref{ft2} implies that there exists a curve
\begin{gather*}
\widetilde{\gamma }(s)=(\gamma (s); X(s),Y(s),T)\in {\rm PSH}(1)
\end{gather*}
such that $\widetilde{\gamma }^{\ast }\omega =\varphi $. By the moving frame formula (\ref{movingframe}), we have
\begin{gather*}
{\rm d}\gamma (s) =X(s){\rm d}s+\tau (s)T{\rm d}s, \qquad {\rm d}X(s) =k(s)Y(s){\rm d}s,\qquad
{\rm d}Y(s) =-k(s)X(s){\rm d}s-T{\rm d}s,
\end{gather*}
which means that
\begin{gather*}
X(s) =\gamma _{\xi}^{\prime}(s),\qquad
k(s) =\left\langle \frac{{\rm d}X(s)}{{\rm d}s},Y(s)\right\rangle ,\qquad
\tau (s) =\left\langle \frac{{\rm d}\gamma (s)}{{\rm d}s},T\right\rangle .
\end{gather*}
This completes the proof of existence.

To prove uniqueness, suppose that two horizontally regular curves $\gamma _{1}$ and $\gamma _{2}$ have the same $p$-curvature $k(s)$ and contact normality $\tau (s)$. The identity (\ref{fodar}) shows that they must have the same Darboux derivatives
\begin{gather*}
\widetilde{\gamma }_{1}^{\ast }\omega =\widetilde{\gamma }_{2}^{\ast }\omega.
\end{gather*}%
Therefore, by Theorem \ref{ft1}, there exists a symmetry $g\in {\rm PSH}(1)$ such that $\widetilde{\gamma }_{2}(s)=g\circ \widetilde{\gamma }_{1}(s)$, and hence $\gamma_{2}(s)=g\circ \gamma _{1}(s)$ for all $s$. This completes the proof of uniqueness up to a group action.

\subsection[The derivation of the $p$-curvature and the contact normality]{The derivation of the $\boldsymbol{p}$-curvature and the contact normality}\label{computofpt}
In the subsection, we will compute the $p$-curvature and the contact normality for horizontally regular curves (Theorem \ref{main2}) and for the geodesics in $H_1$ (Theorem~\ref{chaofgeo}).

\begin{proof}[Proof of Theorem \ref{main2}]
Let $\gamma(t)=(x(t),y(t),z(t))$ be a horizontally regular curve. The ho\-rizontal
arc-length $s$ is def\/ined by
\begin{gather*}
s(t)=\int_{0}^{t}|\gamma_{\xi}^{\prime}(u)|{\rm d}u.
\end{gather*}
We f\/irst observe that there is the natural decomposition
\begin{gather}
\gamma^{\prime}(t)=(x^{\prime}(t),y^{\prime}(t),z^{{\prime}}(t))=x^{\prime}(t)\frac{\partial}{\partial x}+y^{\prime}(t)\frac{\partial}{\partial y}+z^{\prime}(t)\frac{\partial}{\partial z} \nonumber\\
\hphantom{\gamma^{\prime}(t)}{} =
\underbrace{x^{\prime}(t)\mathring{e}_{1}+y^{\prime}(t)\mathring{e}_{2}}_{\gamma'_\xi(t)}+\underbrace{(z^{{\prime}}(t)+xy^{\prime}(t)-yx^{\prime}(t))\frac{\partial}{\partial z}}_{\gamma'_T(t)},\label{velexp}
\end{gather}
where we abuse the notation by $\frac{\partial}{\partial z}=T$. Let $\bar{\gamma}(s)$ be the reparametrization of $\gamma(t)$ by the horizontal arc-length~$s$. Since $\gamma^{\prime}(t)=\bar{\gamma}^{\prime}(s)\frac{{\rm d}s}{{\rm d}t}$, by comparing with the decomposition~(\ref{velexp}), one has
\begin{gather}
\bar{\gamma}_{\xi}^{\prime}(s)=\frac{{\rm d}t}{{\rm d}s}(x^{{\prime}}(t)\mathring{e}_{1}+y^{\prime}(\mathring{e}_{2})), \qquad
\bar{\gamma}_{T}^{\prime}(s)=\frac{{\rm d}t}{{\rm d}s}\big((z^{{\prime
}}(t)+xy^{\prime}(t)-yx^{\prime}(t))T\big).\label{velexp1}
\end{gather}

For the $p$-curvature, by (\ref{velexp1}), note that $X(s)=\frac{{\rm d}t}{{\rm d}s}(x^{\prime}(t)\mathring{e}_{1}+y^{\prime}(t)\mathring{e}_{2})$, and $Y(s)=JX(s)=\frac{{\rm d}t}{{\rm d}s}(x^{\prime}(t)\mathring{e}_{2}-y^{\prime}(t)\mathring{e}_{1})$. A straight-forward computation shows
\begin{gather*}
\frac{{\rm d}X(s)}{{\rm d}s} =\frac{{\rm d}}{{\rm d}s}\left(\frac{{\rm d}t}{{\rm d}s}\left(x^{\prime}(t),y^{\prime}(t),x^{\prime}y(t)-xy^{\prime}(t)\right)\right)\\
\hphantom{\frac{{\rm d}X(s)}{{\rm d}s}}{} =\left(\frac{{\rm d}t}{{\rm d}s}\right)^{2}\left(x^{\prime\prime}(t),y^{{\prime\prime }}(t),x^{\prime\prime}y(t)-xy^{\prime\prime}(t)\right)+
\frac{{\rm d}^{2}t}{{\rm d}s^{2}}\left(x^{\prime}(t),y^{\prime}(t),x^{\prime}y(t)-xy^{\prime}(t)\right) \\
\hphantom{\frac{{\rm d}X(s)}{{\rm d}s}}{}=\left(x^{\prime\prime}(t)\left(\frac{{\rm d}t}{{\rm d}s}\right)^{2}+x^{\prime}(t)\frac{{\rm d}^{2}t}{{\rm d}s^{2}}\right)\mathring{e}_{1}+ \left(y^{\prime\prime}(t)\left(\frac{{\rm d}t}{{\rm d}s}\right)^{2}+y^{\prime}(t)\frac{{\rm d}^{2}t}{{\rm d}s^{2}}\right)\mathring{e}_{2},
\end{gather*}
so
\begin{gather}
k(s)=\left\langle \frac{{\rm d}X(s)}{{\rm d}s},Y(s)\right\rangle \nonumber\\
\hphantom{k(s)}{} =-\left(x^{\prime\prime}(t)\left(\frac{{\rm d}t}{{\rm d}s}\right)^{2}+x^{\prime}(t)\frac{{\rm d}^{2}t}{{\rm d}s^{2}}\right)y^{\prime}(t)\frac{{\rm d}t}{{\rm d}s}
+\left(y^{\prime\prime}(t)\left(\frac{{\rm d}t}{{\rm d}s}\right)^{2}+y^{\prime}(t)\frac{{\rm d}^{2}t}{{\rm d}s^{2}}\right)x^{\prime}(t)\frac{{\rm d}t}{{\rm d}s} \nonumber\\
\hphantom{k(s)}{}=-\left(x^{\prime\prime}(t)y^{\prime}(t)-x^{\prime}(t)y^{\prime\prime}(t)\right)\left(\frac{{\rm d}t}{{\rm d}s}\right)^3
=\frac{x^{\prime}y^{\prime\prime}-x^{\prime\prime}y^{\prime}}{\big((x^{\prime})^{2}+(y^{\prime})^{2}\big)^{\frac{3}{2}}}(t).\label{pcur}
\end{gather}
Again by (\ref{velexp1}), the contact normality must be
\begin{gather}
\tau(s)=\langle \bar{\gamma}^{\prime}(s),T\rangle =\langle \bar{\gamma}_{T}^{\prime}(s),T\rangle=\frac{{\rm d}t}{{\rm d}s}(z^{\prime}(t)+xy^{\prime}(t)-yx^{\prime}(t)) =\frac{xy^{\prime}-x^{\prime}y+z^{\prime}}{\big((x^{\prime})^{2}+(y^{\prime})^{2}\big)^{\frac{1}{2}}}(t),\label{Tvai}
\end{gather}
and the result follows.
\end{proof}

Next we use (\ref{pcur}) and (\ref{Tvai}) to compute the $p$-curvature and the contact normality for the geodesics in $H_1$.

\begin{proof}[Proof of Theorem \ref{chaofgeo}]
Recall \cite{CCG} that the Hamiltonian system on $H_{1}$ for the geodesics is
\begin{gather}\label{HS}
\dot{x}^{k}(t) = h^{kj} (x(t)) \xi_{j}(t), \qquad
\dot{\xi}_{k}(t) = -\frac{1}{2}\sum \limits_{i,j=1}^{3}\frac{\partial h^{ij}(x) }{\partial x^{k}}\xi _{i}\xi _{j},\qquad
k=1,2,3,
\end{gather}
where
\begin{gather*}
h^{ij}\big( x^{1},x^{2},x^{3}\big) =
\begin{pmatrix}
1 & 0 & x^{2} \\
0 & 1 & -x^{1} \\
x^{2} & -x^{1} & \big( x^{1}\big) ^{2}+\big( x^{2}\big) ^{2}%
\end{pmatrix}.
\end{gather*}
So the Hamiltonian system (\ref{HS}) can be expressed by
\begin{gather*}
\dot{x}^{1}(t) =\xi _{1}+x^{2}\xi _{3}, \\
\dot{x}^{2}(t) =\xi _{2}-x^{1}\xi _{3}, \\
\dot{x}^{3}(t) =x^{2}\xi _{1}-x^{1}\xi _{2}+\xi _{3}\big[\big( x^{1}\big) ^{2}+\big( x^{2}\big) ^{2}\big] ,\\
\dot{\xi}_{1}(t) =\xi _{2}\xi _{3}-x^{1}\xi _{3}^{2}, \\
\dot{\xi}_{2}(t) =-\xi _{1}\xi _{3}-x^{2}\xi _{3}^{2}, \\
\dot{\xi}_{3}(t) =0.
\end{gather*}
Since $\dot{\xi}_{3}(t) =0$, we have $\xi _{3}(t)
=c_{3} $ for some constant $c_{3}$. When $c_{3}=0$, one has
$x(t) =( c_{1}t+d_{1},c_{2}t+d_{2},(c_{1}d_{2}-c_{2}d_{1}) t+d_{3})$, and this implies that $k(t) =0$ and
$\tau (t) =0$; when $c_{3}> 0$, one has
\begin{gather}\label{gode1}
x(t) = \big( x^{1}(t) ,x^{2}(t),x^{3}(t) \big),
\end{gather}
where
\begin{gather*}
 x^{1}(t) = a_{1}\sin ( 2c_{3}t) +a_{2}\cos (2c_{3}t) +d_{1}, \\
x^{2}(t) = -a_{2}\sin ( 2c_{3}t) +a_{1}\cos (2c_{3}t) +d_{2}, \\
x^{3}(t) =( a_{2}d_{1}+a_{1}d_{2}) \sin (2c_{3}t) +( a_{2}d_{2}-a_{1}d_{1}) \cos (
2c_{3}t)+2c_{3}\big( a_{1}^{2}+a_{2}^{2}\big) t+d_{3}.
\end{gather*}
Hence $k(t) =-\frac{1}{[ ( a_{1}^{2}+a_{2}^{2}) ] ^{\frac{1}{2}}}<0$ and $\tau (t) =0;$ f\/inally, when $c_{3}< 0$, one has
\begin{gather}\label{gode2}
x(t) = \big( x^{1}(t) ,x^{2}(t),x^{3}(t) \big),
\end{gather}
where
\begin{gather*}
x^{1}(t) = a_{1}\sin (-2c_{3}t) +a_{2}\cos (-2c_{3}t) +d_{1}, \\
x^{2}(t) = a_{2}\sin (-2c_{3}t) -a_{1}\cos (-2c_{3}t) +d_{2}, \\
x^{3}(t) = ( a_{1}d_{1}+a_{2}d_{2}) \sin (-2c_{3}t) -( a_{2}d_{1}-a_{1}d_{2}) \cos (
-2c_{3}t)+2c_{3}\big( a_{1}^{2}+a_{2}^{2}\big) t+d_{3}.
\end{gather*}
Hence $k(t) =\frac{1}{[ ( a_{1}^{2}+a_{2}^{2}) %
] ^{\frac{1}{2}}}>0$ and $\tau (t) =0$.

The calculations above show that a horizontal curve is congruent to a geodesic
if it has positive constant $p$-curvature. Conversely, it is easy to prove
that any geodesic acted by a symmetry is still a geodesic. Therefore we
complete the proof of Theorem \ref{chaofgeo}.
\end{proof}

\begin{Remark}
Actually, the geodesics (\ref{gode1}) for $c_{3}>0$ and (\ref{gode2}) for $c_{3}<0$ travel along the same path with reverse direction.
\end{Remark}

\section[Dif\/ferential invariants of parametrized surfaces in $H_1$]{Dif\/ferential invariants of parametrized surfaces in $\boldsymbol{H_1}$}\label{invpasur}

\subsection{The proof of Theorem \ref{main4}}

Let $F\colon U\rightarrow H_{1}$ be a normal parametrized surface with $a,b,c,l$ and $m$ as the coef\/f\/icients in~(\ref{coeofform}). Denote the lift $\widetilde{F}$ of $F$ to ${\rm PSH}(1)$ as
\begin{gather*}
\widetilde{F}=\langle F(u,v);X(u,v),Y(u,v),T\rangle ,
\end{gather*} where $X(u,v):=F_{u}(u,v)$, $Y(u,v):=JX(u,v)$. As long as $F$ is given, $X=F_u$, $Y=JX$ and the Reeb vector f\/ield $T$ are uniquely determined, and so the lift $\widetilde{F}$ is unique. For convenience, henceforward we denote $F(u,v)$, by $F$, $X(u,v)$ by $X$, and $Y(u,v)$ by $Y$. We begin from deriving the Darboux derivative $\widetilde{F}^{\ast }\omega $ of $\widetilde{F}$.

On one hand, we can write the vector $F_u$ in terms of the frame
\begin{gather}\label{haha11}
{\rm d}F\left(\frac{\partial}{\partial u}\right)=F_u= X+ \langle F_u, Y\rangle Y + \langle F_u, T\rangle T;
\end{gather}actually the last two terms are zero since $F_u=X, Y, T$ are orthonormal.
On the other hand, by the moving frame formula (\ref{movingframe}),
\begin{gather}\label{haha1}
{\rm d}F =X\big(\widetilde{F}^{\ast }\omega ^{1}\big)+Y\big(\widetilde{F}^{\ast }\omega
^{2}\big)+T\big(\widetilde{F}^{\ast }\omega ^{3}\big).
\end{gather}
Apply $\frac{\partial}{\partial u}$ to (\ref{haha1})
\begin{gather*}
{\rm d}F\left(\frac{\partial }{\partial u}\right)=X\big(\widetilde{F}^{\ast }\omega ^{1}\big)\left(\frac{\partial }{\partial u}\right)+Y\big(\widetilde{F}^{\ast }\omega ^{2}\big)\left(\frac{\partial }{\partial u}\right)+T\big(\widetilde{F}^{\ast }\omega ^{3}\big)\left(\frac{\partial }{\partial u}\right),
\end{gather*}
and compare the coef\/f\/icients with those in (\ref{haha11}) we have
\begin{gather}\label{coeff1}
\big(\widetilde{F}^{\ast }\omega ^{1}\big)\left(\frac{\partial }{\partial u}\right)=1,\qquad \text{and} \qquad \big(\widetilde{F}^{\ast }\omega ^{2}\big)\left(\frac{\partial }{\partial u}\right)=\big(\widetilde{F}^{\ast }\omega ^{3}\big)\left(\frac{\partial }{\partial u}\right)=0.
\end{gather}
Similarly by applying $\frac{\partial}{\partial v}$ to (\ref{haha1}), one has
\begin{gather}
\big(\widetilde{F}^{\ast }\omega ^{1}\big)\left(\frac{\partial }{\partial v}\right)
=\langle F_{v},X\rangle =a, \qquad \big(\widetilde{F}^{\ast }\omega ^{2}\big)\left(\frac{\partial }{\partial v}\right)
=\langle F_{v},Y\rangle =b, \nonumber\\
\big(\widetilde{F}^{\ast }\omega ^{3}\big)\left(\frac{\partial }{\partial v}\right) =\langle F_{v},T\rangle =c.\label{coeff2}
\end{gather}
Combine (\ref{coeff1}) and (\ref{coeff2}) to get
\begin{gather}\label{Darder1}
\widetilde{F}^{\ast }\omega ^{1} ={\rm d}u+a{\rm d}v, \qquad
\widetilde{F}^{\ast }\omega ^{2} =b{\rm d}v, \qquad
\widetilde{F}^{\ast }\omega ^{3} =c{\rm d}v.
\end{gather}

To derive $\widetilde{F}^*{\omega_1}^2$, we use (\ref{movingframe}) again and repeat the same process above.
By (\ref{Darder1}),
\begin{gather}
{\rm d}X\left(\frac{\partial}{\partial u}\right)=\underbrace{\langle X_u, X \rangle}_{=0}X +\langle X_u, Y\rangle Y+\langle X_u, T\rangle T =Y\big(\widetilde{F}^{\ast }{\omega _1}^{2}\big)\left(\frac{\partial}{\partial u}\right)+T\big(\widetilde{F}^{\ast}\omega ^{2}\big)\left(\frac{\partial}{\partial u}\right) \nonumber \\
\hphantom{{\rm d}X\left(\frac{\partial}{\partial u}\right)}{} =Y\big(\widetilde{F}^{\ast }{\omega _1}^{2}\big)\left(\frac{\partial}{\partial u}\right),\label{v1}
\end{gather}and so $\langle X_u, T\rangle =0$.
Again,
\begin{gather}
{\rm d}X\left(\frac{\partial}{\partial v}\right)=\underbrace{\langle X_v, X \rangle}_{=0}X +\langle X_v, Y\rangle Y+\langle X_v, T\rangle T =Y\big(\widetilde{F}^{\ast }{\omega _1}^{2}\big)\left(\frac{\partial}{\partial v}\right)+ b T, \label{v2}
\end{gather}one has $\langle X_v, T\rangle = b$.
Since ${\rm d}X={\rm d}F_{u}=F_{uu}{\rm d}u+F_{uv}{\rm d}v$, (\ref{v1}) and (\ref{v2}) imply that
\begin{gather*}
\big(\widetilde{F}^{\ast }{\omega _1}^{2}\big)\left(\frac{\partial }{\partial u}\right)=\langle F_{uu},Y\rangle =l ,\qquad
\big(\widetilde{F}^{\ast }{\omega _{1}}^{2}\big)\left(\frac{\partial }{\partial v}\right)=\langle F_{uv},Y\rangle =m.
\end{gather*}
In conclusion, we have
\begin{gather}\label{Darder2}
{\omega_1}^2=l{\rm d}u+m{\rm d}v,\qquad
b=\langle F_{uv},T\rangle , \nonumber \qquad
0=\langle F_{uu},T\rangle =\langle F_{uv},X\rangle =\langle F_{uu},X\rangle.
\end{gather}
Therefore, by (\ref{Darder1}) and (\ref{Darder2}) we have reached the Darboux derivative
\begin{gather}\label{Darder9}
\widetilde{F}^{\ast }\omega =\left(
\begin{matrix}
0 & 0 & 0 & 0 \\
{\rm d}u+a{\rm d}v & 0 & -l{\rm d}u-m{\rm d}v & 0 \\
b{\rm d}v & l{\rm d}u+m{\rm d}v & 0 & 0 \\
c{\rm d}v & b{\rm d}v & -{\rm d}u-a{\rm d}v & 0%
\end{matrix}
\right) .
\end{gather}
Note that the coef\/f\/icients $a$, $b$, $c$, $l$, $m$ uniquely determine the Darboux derivative in \eqref{Darder9}, and so the proof of uniqueness is completed.

For existence, suppose $a$, $b$, $c$ and $m$, $l$ are functions def\/ined on~$U$. Suppose~$\phi$ is the ${\mathfrak{psh}}(1)$-valued 1-form def\/ined by~\eqref{Darder9}. Then
\begin{gather}\label{Darder10}
{\rm d}\phi =\left(
\begin{matrix}
0 & 0 & 0 & 0 \\
\frac{\partial a}{\partial u} & 0 & \frac{\partial l}{\partial v}-\frac{%
\partial m}{\partial u} & 0 \vspace{1mm}\\
\frac{\partial b}{\partial u} & -\frac{\partial l}{\partial v}+\frac{%
\partial m}{\partial u} & 0 & 0 \vspace{1mm}\\
\frac{\partial c}{\partial u} & \frac{\partial b}{\partial u} & -\frac{%
\partial a}{\partial u} & 0
\end{matrix}
\right) {\rm d}u\wedge {\rm d}v,
\end{gather}
and
\begin{gather}\label{Darder11}
\phi \wedge \phi =\left(
\begin{matrix}
0 & 0 & 0 & 0 \\
-lb & 0 & 0 & 0 \\
al-m & 0 & 0 & 0 \\
-2b & -m+al & bl & 0%
\end{matrix}%
\right) {\rm d}u\wedge {\rm d}v.
\end{gather}
Thus, $\phi $ satisf\/ies the integrability condition ${\rm d}\phi=-\phi \wedge \phi $ if and only if the coef\/f\/icients~$a$,~$b$,~$c$,~$l$ and~$m$ satisfy the
integrability condition~\eqref{intcons1}. Therefore Theorem~\ref{ft2} implies there exists a~map
\begin{gather*}
\widetilde{F}^{\ast }(u,v)= ( F(u,v); X(u,v),Y(u,v),T )
\end{gather*}%
such that $\widetilde{F}^{\ast }\omega =\phi $. Finally, the moving frame formula \eqref{movingframe} implies that $F\colon U\rightarrow H_{1}$ is a map with $a$, $b$, $c$, $l$ and $m$ as the coef\/f\/icients of f\/irst kind and second kind respectively.

\subsection{Invariants of surfaces}

Let $\Sigma \hookrightarrow H_1$ be a surface such that all points on $\Sigma$ are non-singular. For each point $p\in \Sigma $, one can choose a normal parametrization $F\colon (u,v)\in U\rightarrow \Sigma $ around $p$ such that
\begin{gather*}
F_{u}=\frac{\partial F}{\partial u}=X,
\end{gather*}
where $X$ is an unit vector f\/ield def\/ining the characteristic foliation. The following lemma characterizes the normal coordinates.

\begin{lem}\label{norcor}
The normal coordinates are determined up to a transformation of the form
\begin{gather*}
\widetilde{u} =\pm u+g(v), \qquad \widetilde{v} =h(v),
\end{gather*}
for some smooth functions $g(v)$, $h(v)$ with $\frac{\partial h}{\partial v}\neq 0$.
\end{lem}

\begin{proof}
Suppose that $(\widetilde{u},\widetilde{v})$ is any normal coordinates around $p$, i.e.,
\begin{gather*}
F_{\widetilde{u}}=\widetilde{X},
\end{gather*}
where $\widetilde{X}=\pm X$. We have the formula for the change of the coordinates
\begin{gather}
F_{u}=F_{\widetilde{u}}\frac{\partial \widetilde{u}}{\partial u}+F_{\widetilde{v}}\frac{\partial \widetilde{v}}{\partial u}, \qquad
F_{v}=F_{\widetilde{u}}\frac{\partial \widetilde{u}}{\partial v}+F_{\widetilde{v}}\frac{\partial \widetilde{v}}{\partial v}.\label{tlofnorpar}
\end{gather}
Expand $F_{\widetilde{v}}=\widetilde{a}\widetilde{X}+\widetilde{b}\widetilde{
Y}+\widetilde{c}\widetilde{T}$ by the orthonormal basis $\{\widetilde{X},\widetilde{Y},\widetilde{T} \}$. The f\/irst identity of~\eqref{tlofnorpar} implies
\begin{gather}
X=F_u=\widetilde{X}\frac{\partial \widetilde{u}}{\partial u}+\left(\widetilde{a}\frac{\partial \widetilde{v}}{\partial u}\widetilde{X} +\widetilde{b}\frac{%
\partial \widetilde{v}}{\partial u}\widetilde{Y}+\widetilde{c}\frac{\partial\widetilde{v}}{\partial u}\widetilde{T}\right)  =\left(\frac{\partial \widetilde{u}}{\partial u}+\widetilde{a}\frac{\partial \widetilde{v}}{\partial u}\right)\widetilde{X}+ \widetilde{b}\frac{\partial \widetilde{v}}{\partial u}\widetilde{Y}+\widetilde{c}\frac{\partial\widetilde{v}}{\partial u}\widetilde{T}.\!\!\!\!\!\label{tlofnorpar2}
\end{gather}
Since $p$ is a non-singular point, we see that $\widetilde{c}\neq 0$ around $p$, and so
\begin{gather*}
\frac{\partial \widetilde{v}}{\partial u}=0,
\end{gather*}
namely, $ \widetilde{v}=h(v)$
for some function $h(v)$. In addition, comparing the coef\/f\/icient of $X$ in~\eqref{tlofnorpar}, \eqref{tlofnorpar2}, we
have
\begin{gather*}
\pm 1=\frac{\partial \widetilde{u}}{\partial u}+\widetilde{a}\frac{\partial\widetilde{v}}{\partial u}=\frac{\partial \widetilde{u}}{\partial u},
\end{gather*}and hence $\widetilde{u}=\pm u+g(v)$ for some function $g(v)$. Finally we
compute
\begin{gather*}
\det{\left(%
\begin{matrix}
\frac{\partial \widetilde{u}}{\partial u} & \frac{\partial \widetilde{u}}{%
\partial v} \vspace{1mm}\\
\frac{\partial \widetilde{v}}{\partial u} & \frac{\partial \widetilde{v}}{%
\partial v}%
\end{matrix}%
\right)}=\det{\left(%
\begin{matrix}
\pm 1 & \frac{\partial g}{\partial v} \vspace{1mm}\\
0 & \frac{\partial h}{\partial v}
\end{matrix}
\right)}=\pm \frac{\partial h}{\partial v}\neq 0,
\end{gather*}
and the result follows.
\end{proof}

As what we did in \eqref{Darder9}, we can also derive the Darboux derivatives $\widetilde{F}^{*}\omega$ for the normal parametrization. One obtains four 1-forms locally def\/ined on the surface $\Sigma$:
\begin{gather}
{\rm I}=\widetilde{F}^{*}\omega^{1}={\rm d}u+a{\rm d}v, \qquad {\rm II}=\widetilde{F}^{*}\omega^{2}=b{\rm d}v,
\qquad {\rm III}=\widetilde{F}^{*}\omega^{3}=c{\rm d}v, \nonumber \\
{\rm IV}=\widetilde{F}^{*}{\omega_{1}}^{2}=l{\rm d}u+m{\rm d}v,\label{ffs}
\end{gather}
where the functions $a$, $b$, $c$, $m$ and $l$ are def\/ined as~\eqref{coeofform}. Next we show that those 1-forms are invariant under the change of coordinates.

\begin{prop}\label{propofform}
Suppose $\widetilde{\rm I}$, $\widetilde{\rm II}$, $\widetilde{\rm III}$, $\widetilde{\rm IV}$ are those defined as \eqref{ffs} with respect to the other normal coordinates $(\widetilde{u},\widetilde{v})$. Then we have
\begin{gather}\label{invariant14}
\widetilde{\rm I}=\pm {\rm I},\qquad \widetilde{\rm II}=\pm {\rm II},\qquad \widetilde{\rm III}={\rm III}, \qquad \widetilde{\rm IV}={\rm IV}.
\end{gather}
\end{prop}

\begin{proof}
Suppose $\widetilde{a}$, $\widetilde{b}$, $\widetilde{c}$, $\widetilde{l}$, $\widetilde{m}$ are the coef\/f\/icients of f\/irst and second kinds with respect to the normal coordinates $(\widetilde{u}$, $\widetilde{v})$. We point out that all such the coef\/f\/icients have the same expression as in~\eqref{coeofform} w.r.t.\ the new coordinates except for $\widetilde{X}=\pm X$ and $\widetilde{Y}=J\widetilde{X}=\pm Y$.

By Lemma~\ref{norcor}, there exists two functions $g(v)$ and $h(v)$ such that
\begin{gather*}
\widetilde{u}=\pm u+g(v), \qquad \widetilde{v}=h(v).
\end{gather*}
Now we compute the transformation laws of those coef\/f\/icients:
\begin{gather}
a=\langle F_{v},X\rangle =\left\langle F_{\widetilde{u}}\frac{\partial \widetilde{u}}{\partial v}+F_{\widetilde{v}}\frac{\partial \widetilde{v}}{\partial v},X\right\rangle=\left\langle \pm X\frac{\partial g}{\partial v}+F_{\widetilde{v}}\frac{\partial h}{\partial v},X\right\rangle
=\pm \left(\frac{\partial g}{\partial v}+\frac{\partial h}{\partial v}
\widetilde{a}\right).\label{tlofc}
\end{gather}
Similarly, we have
\begin{gather} \label{tlofc1}
b=\pm \frac{\partial h}{\partial v}\widetilde{b},\qquad c=\frac{\partial h}{\partial v}\widetilde{c},
\end{gather}
and so
$F_{u}=\pm F_{\widetilde{u}}$ and $F_{uu}=\pm\big(F_{\widetilde{u}\widetilde{u}}\frac{\partial \widetilde{u}}{\partial u}+F_{\widetilde{u}\widetilde{v}}\frac{\partial \widetilde{v}}{\partial u}\big)=F_{\widetilde{u}\widetilde{u}}$. Thus
\begin{gather} \label{tlofc2}
l=\pm \widetilde{l}.
\end{gather}
Similarly
\begin{gather} \label{tlofc3}
m=\frac{\partial g}{\partial v}\widetilde{l}+\frac{\partial h}{\partial v}\widetilde{m}.
\end{gather}
From the transformation laws \eqref{tlofc}, \eqref{tlofc1}, \eqref{tlofc2}, \eqref{tlofc3}, the result \eqref{invariant14} follows.
\end{proof}

\begin{Remark}
In the proof \eqref{tlofc1}, denote
\begin{gather}\label{alpha1}
\alpha=\frac{b}{c}, \qquad \widetilde{\alpha}=\frac{\widetilde{b}}{\widetilde{c}},
\end{gather} then we have $\alpha=\pm\widetilde{\alpha}$. Actually, $\alpha$ is a function def\/ined on the non-singular part of $\Sigma$, independent of the choice of the normal coordinates up to a sign, such that $\alpha e_{2}+T\in T\Sigma$, and hence an invariant of $\Sigma$ on the non-singular part.
Similarly, from~\eqref{tlofc2}, so is for~$l$, which actually is the $p$-mean curvature.
\end{Remark}

\begin{Remark}
We point out that the signs appearing for $\alpha$ and $l$ are due to the dif\/ferent choices of the orientations. Indeed, if one chooses the normal coordinates with respect to a f\/ixed orientation of the characteristic foliation, then we have $\alpha =\widetilde{\alpha }$ and $l=\widetilde{l}$.
\end{Remark}

Besides the invariants $\alpha$ and $l$, we now proceed to the other invariant of $\Sigma$. Actually, by Proposition~\ref{propofform}, we have
\begin{gather*}
{\rm I}\otimes {\rm I}+{\rm II}\otimes {\rm II}+{\rm III}\otimes {\rm III}=\widetilde{\rm I}\otimes \widetilde{\rm I}+\widetilde{\rm II}\otimes \widetilde{\rm II}+\widetilde{\rm III}\otimes \widetilde{\rm III}.
\end{gather*}
Therefore the dif\/ferential form ${\rm I}\otimes {\rm I}+{\rm II}\otimes {\rm II}+{\rm III}\otimes {\rm III}$ again is independent of the choices of the normal coordinates, and hence an invariant of~$\Sigma$. Next we characterize this invariant.

\begin{lem}\label{lastinv}
Let $g_{\Theta}$ be the adapted metric on $H_1$. Then we have
\begin{gather*}
g_{\Theta}|_{\Sigma}={\rm I}\otimes {\rm I}+{\rm II}\otimes {\rm II}+{\rm III}\otimes {\rm III}
\end{gather*}
defined on the non-singular part of $\Sigma$.
\end{lem}

\begin{proof}This lemma is a direct consequence of \eqref{ffs}, \eqref{movingframe}, and $g_{\Theta}|_{\Sigma}={\rm d}p\otimes {\rm d}p$.
\end{proof}

Finally we mention that although the $1$-forms I, II, III, IV are only def\/ined on the non-singular points, the invariant ${\rm I}\otimes {\rm I}+{\rm II}\otimes {\rm II}+{\rm III}\otimes {\rm III}$ can be smoothly extended to the whole surface $\Sigma$ by Lemma~\ref{lastinv}.

\subsection[A complete set of invariants for surfaces in $H_1$]{A complete set of invariants for surfaces in $\boldsymbol{H_1}$}

In this section, we will obtain the last invariant ${\rm IV}=\widetilde{F}^{\ast}{\omega _{1}}^{2}$, which is completely determined by the invariants $\alpha$, $g_{\Theta }$, $l$. We therefore have a complete set of invariants for the non-singular part of the surfaces in~$H_1$.

Let $f\colon\Sigma \rightarrow H_{1}$ be an embedding oriented surface in~$H_1$. For convenience, we will not distinguish the surfaces $\Sigma $ and $f(\Sigma )$. At any non-singular point $p\in\Sigma $, we choose the orthonormal frame $(p;e_{1},e_{2},T)$, where $e_{1} $ is tangent to the characteristic foliation and $e_{2}=Je_{1}$. A Darboux frame is a moving frame which is smoothly def\/ined on $\Sigma $ except for the singular points, and hence there exists a lifting of $f$ to ${\rm PSH}(1)$ def\/ined by~$F$. Now we would like to compute the Darboux derivative $F^{\ast }\omega $ of~$F$. In the following, we abuse the notation $F^{\ast }\omega $ by taking
\begin{gather*}
\omega =\left(
\begin{matrix}
0 & 0 & 0 & 0 \\
\omega ^{1} & 0 & -\omega _{1}{}^{2} & 0 \\
\omega ^{2} & \omega _{1}{}^{2} & 0 & 0 \\
\omega ^{3} & \omega ^{2} & -\omega ^{1} & 0%
\end{matrix}
\right) ,
\end{gather*}
to \looseness=-1 express the Darboux derivative. It satisf\/ies the integrability condition
${\rm d}\omega +\omega \wedge \omega =0$, that is,
\begin{gather}\label{mcsteq}
{\rm d}\omega ^{1} ={\omega _{1}}^{2}\wedge \omega ^{2}, \qquad
{\rm d}\omega ^{2} =-{\omega _{1}}^{2}\wedge \omega ^{1}, \qquad
{\rm d}\omega ^{3} =2\omega ^{1}\wedge \omega ^{2}, \qquad
{\rm d}{\omega _{1}}^{2} =0.
\end{gather}

Let $g_{\Theta}=h+\Theta^{2}$ be the adapted metric. By \eqref{ffs} we know $\omega^{2}=\alpha \omega^{3}$ on the non-singular part of $\Sigma$, and it is easy to see that
\begin{gather*}
g_{\Theta}|_{\Sigma}=\omega^{1}\otimes \omega^{1}+\omega^{2}\otimes\omega^{2}+\omega^{3}\otimes \omega^{3}=\omega^{1}\otimes \omega^{1}+\big(1+\alpha^2\big)\omega^{3}\otimes \omega^{3}.
\end{gather*}
Set
\begin{gather} \label{baeq1}
\hat{\omega}^{1}=\omega^{1}, \qquad
\hat{\omega}^{2}=\sqrt{1+\alpha^{2}}\omega^{3},
\end{gather}
which \looseness=-1 form an orthonormal coframe on $\Sigma$ w.r.t.\ the metric $g_{\Theta}|_{\Sigma}$; the corresponding dual frame is
\begin{gather*}
\hat{e}_{1}=e_{1}, \qquad
\hat{e}_{2}=e_{\Sigma}=\frac{\alpha e_{2}+T}{\sqrt{1+\alpha^{2}}}.
\end{gather*}
If ${\hat{\omega}_1}{}^{2}$ is the Levi-Civita connection of $g_{\Theta}|_{\Sigma}$ with respect to the coframe $\hat{\omega}^{1}$, $\hat{\omega}^{2}$, by the fundamental theorem in Riemannian geometry, we have the structure equations
\begin{gather} \label{risteq}
{\rm d}\hat{\omega}^{1}=-\hat{\omega}_{2}{}^{1}\wedge \hat{\omega}^{2}, \qquad
{\rm d}\hat{\omega}^{2}=-\hat{\omega}_{1}{}^{2}\wedge \hat{\omega}^{1}, \qquad
\hat{\omega}_{1}{}^{2}=-\hat{\omega}_{2}{}^{1}.
\end{gather}

The following proposition shows that ${\omega_{1}}^{2}$ is completely determined by the induced f\/irst fundamental form $g_{\Theta}|_{\Sigma}$ and
the functions $\alpha$ and $l$ def\/ined in \ref{alpha1}.

\begin{prop}\label{prop1}
We have
\begin{gather*}
{\omega_{1}}^{2}=\frac{\alpha}{\sqrt{1+\alpha^{2}}}\hat{\omega}_{1}{}^{2}+\frac{l}{1+\alpha^{2}}\hat{\omega}^{1}+ \frac{e_{1}\alpha}{(1+\alpha^2)^{\frac{3}{2}}}\hat{\omega}^2=l\hat{\omega}^1+\frac{2\alpha^{2}+(e_{1}\alpha)}{\sqrt{1+\alpha^{2}}}\hat{\omega}^2, \\
\hat{\omega}_{1}{}^{2}=\frac{\alpha}{\sqrt{1+\alpha^{2}}}{\omega_{1}}^{2}+\frac{2\alpha}{1+\alpha^2}\hat{\omega}^{2}
=\frac{l\alpha}{\sqrt{1+\alpha^{2}}}\hat{\omega}^{1}+\left(2\alpha+\frac{
\alpha(e_{1}\alpha)}{1+\alpha^{2}}\right)\hat{\omega}^{2}.
\end{gather*}
\end{prop}

\begin{proof}By $\omega ^{2}=\alpha \omega ^{3}$ and the second identity of \eqref{baeq1}, we have
\begin{gather*}
{\rm d}\omega ^{2} ={\rm d}\left( \frac{\alpha }{(1+\alpha ^{2})^{\frac{1}{2}}}\hat{\omega}^{2}\right) ={\rm d}\left( \frac{\alpha }{(1+\alpha ^{2})^{\frac{1}{2}}}\right) \wedge \hat{\omega}^{2}+\frac{\alpha }{(1+\alpha ^{2})^{\frac{1}{2}}}{\rm d}\hat{\omega}^{2} \\
\hphantom{{\rm d}\omega ^{2}}{} =e_{1}\left( \frac{\alpha }{(1+\alpha ^{2})^{\frac{1}{2}}}\right) \hat{\omega}^{1}\wedge \hat{\omega}^{2}-\frac{\alpha }{(1+\alpha ^{2})^{\frac{1}{2
}}}\hat{\omega}_{1}{}^{2}\wedge \hat{\omega}^{1} \\
\hphantom{{\rm d}\omega ^{2}}{} =\hat{\omega}^{1}\wedge \left( e_{1}\left( \frac{\alpha }{(1+\alpha ^{2})^{\frac{1}{2}}}\right) \hat{\omega}^{2}+\frac{\alpha }{(1+\alpha ^{2})^{\frac{1}{2}}}\hat{\omega}_{1}{}^{2}\right),
\end{gather*}
where we have used the second formula of the structure equation \eqref{risteq} at the third equality above. On the other hand,
from the Maurer--Cartan structure equation \eqref{mcsteq}
\begin{gather*}
{\rm d}\omega ^{2}=-{\omega _{1}}^{2}\wedge \omega ^{1}=\hat{\omega}^{1}\wedge {\omega _{1}}^{2}.
\end{gather*}
Combine two identities above and use the Cartan lemma, we see that there exists a function $D$ such that
\begin{gather}
\omega _{1}{}^{2}
 =\frac{e_{1}\alpha }{(1+\alpha ^{2})^{\frac{3}{2}}}\hat{\omega}^{2}+\frac{\alpha }{(1+\alpha ^{2})^{\frac{1}{2}}}\hat{\omega}_{1}{}^{2}+D\hat{\omega}^{1}.
\label{baeq2}
\end{gather}
Similarly,
\begin{gather*}
-\hat{\omega}_{2}{}^{1}\wedge \hat{\omega}^{2}={\rm d}\hat{\omega}^{1}={\rm d}\omega ^{1} ={\omega _{1}}^{2}\wedge \omega ^{2}=\frac{\alpha }{\sqrt{1+\alpha ^{2}}}{\omega _{1}}^{2}\wedge \hat{\omega}^{2}.
\end{gather*}
Again, by Cartan lemma, there exists a function $A$ such that
\begin{gather}\label{baeq3}
-\hat{\omega}_{2}{}^{1}=\frac{\alpha }{\sqrt{1+\alpha ^{2}}}\omega _{1}{}^{2}+A\hat{\omega}^{2}.
\end{gather}
Finally, use \eqref{risteq} again
\begin{gather*}
-\hat{\omega}_{1}{}^{2}\wedge \hat{\omega}^{1} ={\rm d}\hat{\omega}^{2}={\rm d}\big(\big(1+\alpha ^{2}\big)^{\frac{1}{2}}\omega ^{3}\big)=\big(1+\alpha ^{2}\big)^{\frac{1}{2}}{\rm d}\omega ^{3}+{\rm d}\big(1+\alpha ^{2}\big)^{\frac{1}{2}}\wedge \omega ^{3} \\
\hphantom{-\hat{\omega}_{1}{}^{2}\wedge \hat{\omega}^{1}}{} =2\alpha \big(1+\alpha ^{2}\big)^{\frac{1}{2}}\hat{\omega}^{1}\wedge \omega ^{3}+\frac{\alpha }{(1+\alpha ^{2})^{\frac{1}{2}}}{\rm d}\alpha \wedge \omega ^{3}=\left( 2\alpha +\frac{\alpha (e_{1}\alpha )}{1+\alpha ^{2}}\right) \hat{\omega}^{1}\wedge \hat{\omega}^{2},
\end{gather*}
where we have used the third formula of \eqref{mcsteq} and $\hat{\omega}%
^{2}\wedge \omega ^{3}=0$. Therefore, there exists a~func\-tion~$B$ such that
\begin{gather}\label{baeq4}
\hat{\omega}_{1}{}^{2}=\left( 2\alpha +\frac{\alpha (e_{1}\alpha )}{1+\alpha^{2}}\right) \hat{\omega}^{2}+B\hat{\omega}^{1}.
\end{gather}
By \eqref{baeq2}, \eqref{baeq3}, we get
\begin{gather*}
D=\omega _{1}{}^{2}(e_{1})-\frac{\alpha }{\sqrt{1+\alpha ^{2}}}\hat{\omega}_{1}{}^{2}(e_{1})=\frac{\omega_{1}{}^{2}(e_{1})}{1+\alpha ^{2}}=\frac{l}{1+\alpha ^{2}}.
\end{gather*}
Similarly, by \eqref{baeq2}, \eqref{baeq3}, \eqref{baeq4}, we obtain
\begin{gather*}
A =\frac{2\alpha }{1+\alpha ^{2}}, \qquad B =\frac{l\alpha }{\sqrt{1+\alpha ^{2}}}.
\end{gather*}
These complete the proof.
\end{proof}

\section{The derivation of the integrability condition (\ref{Intconsur})}\label{section6}
\begin{proof}
Proposition~\ref{prop1} implies that
\begin{gather*}
0={\rm d}\omega_{1}{}^{2} ={\rm d}\left(\frac{\alpha}{\sqrt{1+\alpha^{2}}}\hat{\omega}_{1}{}^{2}+\frac{l}{%
1+\alpha^{2}}\hat{\omega}^{1}+ \frac{e_{1}\alpha}{(1+\alpha^2)^{\frac{3}{2}}}%
\hat{\omega}^2\right) \\
\hphantom{0}{} =\big \{{-}\big(1+\alpha^{2}\big)^{\frac{3}{2}}(e_{\Sigma}l)+\big(1+\alpha^{2}\big)(e_{1}e_{1}\alpha)-\alpha(e_{1}\alpha)^{2} +4\alpha\big(1+\alpha^{2}\big)(e_{1}\alpha) \\
\hphantom{0=}{} +\alpha\big(1+\alpha^{2}\big)^{2}K+\alpha l\big(1+\alpha^{2}\big)^{\frac{1}{2}
}(e_{\Sigma}\alpha)+\alpha\big(1+\alpha^{2}\big)l^{2}\big \} \frac{\hat{\omega}
^{1}\wedge \hat{\omega}^{2}}{(1+\alpha^2)^{\frac{5}{2}}}.
\end{gather*}
Therefore the integrability condition \eqref{Intconsur} is equivalent to ${\rm d}\omega_{1}{}^{2}=0$.
\end{proof}

\section{The proof of Theorem \ref{main8}}\label{section7}
\begin{proof}
First we show existence. Def\/ine a ${\mathfrak{psh}}(1)$-valued 1-form $\phi$ on the non-singular part of $\Sigma$ by
\begin{gather*}
\phi=\left(
\begin{matrix}
0 & 0 & 0 & 0 \\
\hat{\omega}^{1} & 0 & -\omega_{1}{}^{2} & 0 \\
\frac{\alpha^{\prime}}{\sqrt{1+(\alpha^{\prime})^{2}}}\hat{\omega}%
^{2} & \omega_{1}{}^{2} & 0 & 0 \\
\frac{1}{\sqrt{1+(\alpha^{\prime})^{2}}}\hat{\omega}^{2} & \frac{%
\alpha^{\prime}}{\sqrt{1+(\alpha^{\prime})^{2}}}\hat{\omega}^{2} & -%
\hat{\omega}^{1} & 0%
\end{matrix}
\right),
\end{gather*}
where
\begin{gather*}
\omega_{1}{}^{2}=\frac{\alpha^{\prime}}{\sqrt{1+(\alpha^{\prime})^{2}}}\hat{\omega}_{1}{}^{2}+\frac{l^{\prime}}{1+(\alpha^{\prime})^{2}}\hat{%
\omega}^{1}+ \frac{e_{1}\alpha^{\prime}}{(1+(\alpha^{\prime})^{2})^{\frac{3}{2}}}\hat{\omega}^2.
\end{gather*}
It is easy to check that $\phi$ satisf\/ies ${\rm d}\phi+\phi \wedge \phi=0$ if and only if the integrability condition \eqref{Intconsur} holds. Therefore,
by Theorem~\ref{ft2}, for any point $p\in \Sigma$, there exists an open set~$U$ containing~$p$ and an embedding $f\colon U\rightarrow H_1$ such that $g=f^{*}(g_{\Theta})$, $\alpha^{\prime}=f^{*}\alpha$ and $l^{\prime}=f^{*}l$.

For uniqueness, by Proposition~\ref{prop1}, the Darboux derivative is completely determined by the induced metric $%
g_{\Theta}|_{\Sigma}$, the $p$-variation $\alpha$ and the $p$-mean
curvature $l$. Therefore, by Theorem~\ref{ft1}, the embedding into $H_1$ is
unique up to a Heisenberg rigid motion.
\end{proof}

\section{Application: the Crofton formula}

Since the singular set of a $C^2$-surface in $H_1$ consists of only isolated points or singular curves \cite[Theorem B]{CHMY2}, and the integral of the intersections of horizontal lines and the surface over the singular set has zero measure, we may assume that $\Sigma$ is a $C^2$-surface without singular points throughout this section.

\begin{Definition}An \textit{oriented horizontal line} $\ell$ in $H_1$ is an oriented line such that any point $p\in \ell$ the tangent vector of the line at $p$ lies on the contact plane $\xi_p$. For convenience we sometimes call a \textit{horizontal line} or a \textit{line}. Denote $\mathcal{L}$ by the set of all oriented horizontal lines in~$H_1$.
\end{Definition}

\begin{prop}\label{linerepre}
Any horizontal line $\ell \in \mathcal{L}$ can be parametrized by a triple $(p,\theta, t)\in \mathbb{R} \times S^1\times\mathbb{R}$, and also be parametrized by a base point $B=(p\cos\theta, p\sin\theta, t)$ with a horizontally unit-speed vector $U=(\sin\theta, -\cos\theta, p) $, namely,
\begin{gather}\label{parametrizedline}
\ell(s)\colon \ (p\cos\theta, p\sin\theta, t)+ s(\sin\theta,- \cos\theta, p), \qquad \forall \, s\in\mathbb{R}.
\end{gather}
\end{prop}

\begin{proof}
Consider the projection $\pi(\ell)$ of the line $l\in\mathcal{L}$ onto the $xy$-plane. Since $\pi(\ell)$ can be uniquely determined by the pair $(p,\theta)$, where $p\in\mathbb{R}$ is the oriented distance from the origin to the line $\pi(l)$ (see \cite{Ch} or the remark below) and $\theta\in [0,2\pi)$ is the angle from the positive $x$-axis to the normal (Fig.~\ref{fig:1.1}), the points $(x,y)\in\pi(\ell)$ satisfy the equation
\begin{gather}\label{projformula}
x\cos\theta + y\sin\theta = p.
\end{gather}

\begin{figure}[ht] \centering
 \includegraphics[scale=0.6]{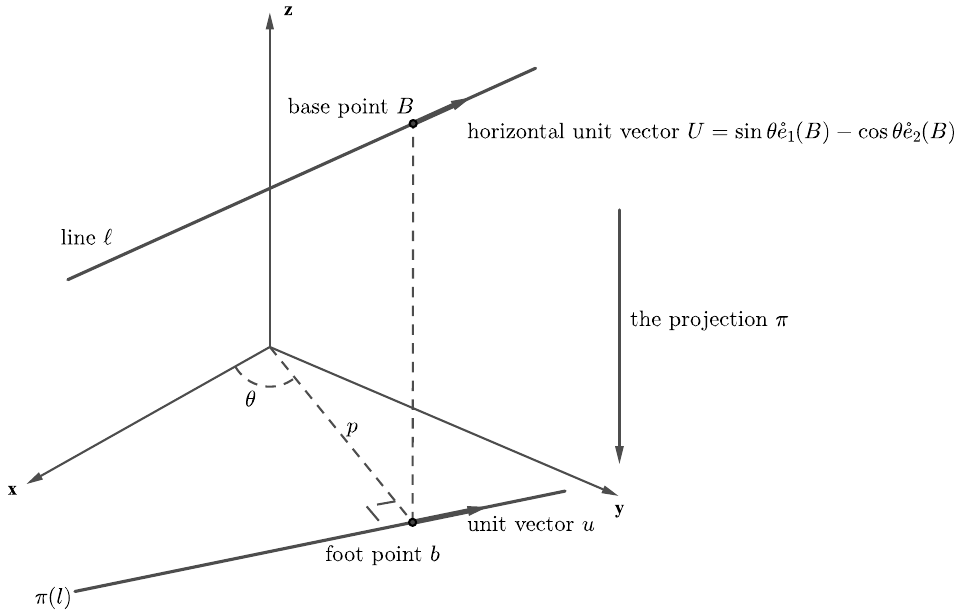}
\caption{} \label{fig:1.1}
\end{figure}

On the projection $\pi(\ell)$, denote the foot point
\begin{gather*}
b=(p\cos\theta, p\sin\theta),
\end{gather*}
and the unit tangent vector along the projection
\begin{gather}\label{xyunit}
u=(\sin\theta, -\cos\theta), \qquad |u|_{\mathbb{R}^2}=1,
\end{gather}
where $|u|_{\mathbb{R}^2}$ is the Euclidean length of $u$ on the $xy$-plane; on the line $\ell\in H_1$, denote the lifting of the foot point $b$, called the \textit{base point} by
\begin{gather*}
B=(p\cos\theta, p\sin\theta,t) \qquad \text{for some} \quad t\in \mathbb{R}.
\end{gather*}
Denote the tangent vector of $\ell$ at point $B$ by $T(B)$. Since $\ell$ is horizontal, which implies that $T(B)\in\xi_B:=\operatorname{span}\{\mathring{e}_1(B), \mathring{e}_2(B)\}$, and we have
\begin{gather}
T(B)=\alpha \mathring{e}_1(B)+\beta\mathring{e}_2(B)=\alpha(1,0,p\sin\theta)+\beta(0,1,-p\cos\theta) \nonumber \\
\hphantom{T(B)}{} =(\alpha, \beta, \alpha p\sin\theta-\beta p\cos\theta) \label{horizontalunit}
\end{gather}for some $\alpha, \beta\in \mathbb{R}$.
Notice that the projection $\pi(T(B))$ is exactly the unit tangent vector $u$ along the projection $\pi(\ell)$. Hence by comparing the f\/irst two components of \eqref{horizontalunit} with \eqref{xyunit} we have
\begin{gather*}
\alpha = \sin\theta, \qquad \beta = -\cos\theta, \qquad \text{and} \qquad T(B)=(\sin\theta,-\cos\theta, p).
\end{gather*}
Therefore by def\/ining the horizontal vector
\begin{gather*}
U:=T(B)=\sin\theta \mathring{e}_1(B)-\cos\theta\mathring{e}_2(B),
\end{gather*}we have $|U|_{\xi(B)}=1$, the horizontally unit-speed, and conclude that the line $\ell$ can be uniquely determined by the triple $(p,\theta, t)$, i.e., the base point $B$, and be parametrized by $B+sU$ for any $s\in \mathbb{R}$ as shown in~\eqref{parametrizedline}.
\end{proof}

\begin{Remark}
We point out that the lines we consider in $H_1$ are all \textit{oriented} lines. Indeed, by convention in $\mathbb{R}^2$ there exists a bijection between the set of oriented lines and $\mathbb{R}^2\times S^1$, and the orientation of $\ell\in H_1$ follows that of $\pi(\ell)\in \mathbb{R}^2$. If we consider the non-oriented lines in $\mathbb{R}^2$ (and hence in $H_1$), then the coef\/f\/icient on the right-hand side of \eqref{croftonidentity} should be changed to $2$.
\end{Remark}

Next, we consider the intersections of lines and a f\/ixed surface
\begin{gather*}
\mathbb{X}\colon \ (u,v)\in\Omega\rightarrow (x(u,v),y(u,v),z(u,v))\in\Sigma
\end{gather*}
embedded in $H_1$ for some domain $\Omega\subset \mathbb{R}^2$. To describe the position of the intersection in~$\mathbb{R}^3$, one needs exact three variables. We have already known, by Proposition~\ref{linerepre}, a line can be represented by a triple $(p,\theta,t)$. Hence if we regard lines and surfaces as a whole system (the conf\/iguration space) and use f\/ive variables $\{(p,\theta, t, u, v )\}$ to describe the behavior of the intersections, two additional constraints are necessarily required to make the number of the freedoms be three. Those constraints can be obtained from the following proposition.

\begin{prop}Let $\mathbb{X}(u,v)=(x(u,v),y(u,v),z(u,v))\in \Sigma$ be the parametrized surface in~$H_1$. Then the conf\/iguration space $D$ which describes the horizonal oriented lines intersecting $\Sigma$ should be
\begin{gather*}
D:=\{(p,\theta, t, u, v)\in \mathbb{R}\times S^1\times \mathbb{R} \times \Omega\\
\hphantom{D:=\{}{} |\,\text{the lines } (p,\theta, t)\in \mathcal{L} \text{ pass through the point }\mathbb{X}(u,v)\text{ on } \Sigma \}\\
\hphantom{D}{} =\{(p,\theta, t, u, v)\in \mathbb{R}\times S^1\times \mathbb{R} \times \Omega
\, |\, \text{the variables } p,\,\theta, \,t,\, u, \, v \text{ satisfy } \eqref{constraint1} \text{ and } \eqref{constraint2}\},
\end{gather*}where
\begin{gather}
x(u,v)\cos\theta+y(u,v)\sin\theta=p, \label{constraint1}\\
z(u,v)=t+(x(u,v)\sin\theta-y(u,v)\cos\theta)p. \label{constraint2}
\end{gather}
\end{prop}

\begin{proof}
Suppose the line $\ell(s)$ parametrized by \eqref{parametrizedline} intersects the surface $\Sigma$ at the point $q$. 
\begin{figure}[ht]
 \centering
 \includegraphics[scale=0.6]{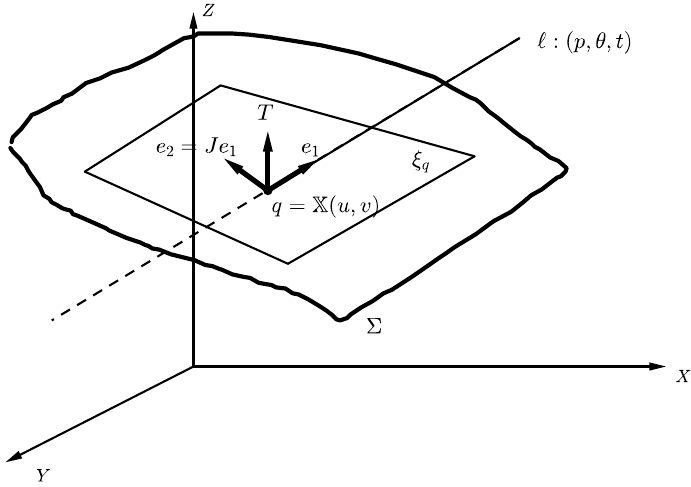}
\caption{} \label{fig:1.2}
\end{figure}

At the point $q$, by Proposition~\ref{linerepre}, we have
\begin{gather}
x(u,v)= p\cos\theta + s\cdot \sin\theta, \label{x(uv)}\\
y(u,v)= p\sin\theta - s\cdot \cos\theta, \label{y(uv)}\\
z(u,v)= t +s\cdot p, \label{z(uv)}
\end{gather} for some $s\in \mathbb{R}$.
By \eqref{x(uv)}, \eqref{y(uv)}, one has
\begin{gather*}
 x(u,v)\cos\theta+y(u,v)\sin\theta=p,
\end{gather*}
which is compatible with $\eqref{projformula}$ and we obtain the f\/irst constraint~\eqref{constraint1}. Finally, use~\eqref{x(uv)},~\eqref{y(uv)} again to solve for the parameter $s$, and substitute $s$ into $\eqref{z(uv)}$. It is easy to have the second constraint~\eqref{constraint2}
\end{proof}

\begin{Remark}\label{shift}
By a simple calculation and \eqref{constraint1}, we observe that
\begin{gather*}
U(B)=\sin\theta\mathring{e}_1(B)-\cos\theta\mathring{e}_2(B)
= \sin\theta\mathring{e}_1(\mathbb{X}(u,v))-\cos\theta\mathring{e}_2(\mathbb{X}(u,v))=U(\mathbb{X}(u,v)),
\end{gather*} i.e., the horizontally unit-speed vector f\/ield $U$ along the line have the same vector-value wherever being evaluated at the based point $B$ or at the intersection $q=\mathbb{X}(u,v)$.
\end{Remark}

Actually, the coordinates $(u,v)$ determine where the intersections should be located on the surface, and the angle $\theta$ decides how those lines penetrate through the surface. Thus, instead of using $(p,\theta,t)$ as the coordinates for the conf\/iguration space, we can also take the triple $\{(u,v,\theta)\in\Omega\times S^1\}$ as the coordinates. Since the intersection $q$ is not only on the line but on the surface, we can derive the change of the coordinates for those coordinates.

By Remark~\ref{shift} we choose the frame
 $\{\mathbb{X}(u,v);e_1(\theta),e_2(\theta),T \}$ on $D$ where (Fig.~\ref{fig:1.2})
\begin{gather}\label{frames}
e_1 := \sin\theta \mathring{e}_1-\cos\theta\mathring{e}_2, \qquad
e_2 := J e_1= \cos\theta \mathring{e}_1+\sin\theta\mathring{e}_2, \qquad
T:=(0,0,1),
\end{gather}
and denote the corresponding coframe $\{\mathbb{X}(u,v); \omega^1, \omega^2, \Theta\}$ with the connection $1$-form~${\omega_1}^2$. The following formula connects the coordinates $(p,\theta, t)$ of the line and the coframe.

\begin{prop}\label{changofcoordinates1}
Let $\{\mathbb{X}(u,v);e_1(\theta),e_2(\theta),T \}$ be a frame defined by~\eqref{frames} and the correspon\-ding coframe $\{\mathbb{X}(u,v); \omega^1, \omega^2, \Theta\}$ with the connection $1$-form~${\omega_1}^2$. We have
\begin{gather*}
\omega^2={\rm d}p+\langle \mathbb{X},e_1\rangle {\rm d}\theta,\qquad
{\omega_1}^2 ={\rm d}\theta, \qquad
\Theta ={\rm d}t, \quad \text{\rm mod} \ {\rm d}\theta, {\rm d}p.
\end{gather*}
One concludes that
\begin{gather}
\omega^2\wedge {\omega_1}^2 = {\rm d}p\wedge {\rm d}\theta, \qquad
\omega^2\wedge {\omega_1}^2 \wedge \Theta= {\rm d}p\wedge {\rm d}\theta\wedge {\rm d}t=\pi^*{\rm d}L, \label{otop}
\end{gather}where $\pi$ is the projection from $D$ to $\mathcal{L}$, and $\langle \,,\,\rangle $ is the Levi-metric.
\end{prop}

\begin{proof} On the surface since $\mathbb{X}=(x,y,z)=x(1,0,y)+y(0,1,-x)+(0,0,z)=x\mathring{e}_1+y\mathring{e}_2+ z T$, we have
\begin{gather*}
\langle \mathbb{X},e_1\rangle =\langle x\mathring{e}_1+y\mathring{e}_2+ z T, \sin\theta\mathring{e}_1-\cos\theta\mathring{e}_2\rangle =x\sin\theta-y\cos\theta.
\end{gather*}
Thus, by the moving frame formula \eqref{movingframe} and the f\/irst constraint \eqref{constraint1}
\begin{gather*}
\omega^2 =\langle {\rm d}\mathbb{X}, e_2\rangle =\left\langle {\rm d}x\, \mathring{e}_1+{\rm d}y\, \mathring{e}_2+\Theta \frac{\partial}{\partial z},e_2 \right\rangle=\cos\theta {\rm d}x +\sin\theta {\rm d}y \\
\hphantom{\omega^2}{} ={\rm d}p+ (x \sin\theta-y \cos\theta){\rm d}\theta={\rm d}p+\langle \mathbb{X},e_1\rangle {\rm d}\theta;\\
{\omega_1}^2=-{\omega_2}^1= -\langle {\rm d}e_2, e_1\rangle =\langle \sin\theta {\rm d}\theta \, \mathring{e}_1+\cos\theta {\rm d}\theta \, \mathring{e}_2, \sin\theta\mathring{e}_1+\cos\theta\mathring{e}_2\rangle \\
\hphantom{{\omega_1}^2}{} =\sin^2\theta\, {\rm d}\theta + \cos^2\theta \, {\rm d}\theta ={\rm d}\theta.
\end{gather*}
By the second constraint \eqref{constraint2} and the parametrization of the line \eqref{parametrizedline}
\begin{gather*}
\Theta={\rm d}z+x{\rm d}y-y{\rm d}x =\big({\rm d}t+(x \sin\theta-y\cos\theta){\rm d}p+p{\rm d}(x \sin\theta-y \cos\theta)\big)+x{\rm d}y-y{\rm d}x\\
\hphantom{\Theta}{} ={\rm d}t+(p\sin\theta-y){\rm d}x-(p\ \cos\theta-x){\rm d}y, \quad \text{mod} \ {\rm d}\theta, {\rm d}p \\
\hphantom{\Theta}{}={\rm d}t+s( \cos\theta {\rm d}x + \sin\theta {\rm d}y), \quad \text{mod} \ {\rm d}\theta, {\rm d}p,\quad \text{for some} \ s\in \mathbb{R}\\
\hphantom{\Theta}{} ={\rm d}t, \quad \text{mod} \ {\rm d}\theta, {\rm d}p,
\end{gather*}
and the result follows.
\end{proof}

The next lemma characterizes the $1$-dimension foliation.
\begin{lem}\label{lemmaiff}
Let $E=\alpha\mathbb{X}_u + \beta\mathbb{X}_v$ be the tangent vector field defined on the surface $\sum=\mathbb{X}(u,v)$. Then the vector $E$ is on the contact bundle $\xi$ $($and hence in $TH_1\cap\xi)$ if and only if pointwisely the coefficients $\alpha$ and $\beta$ satisfy
\begin{gather}\label{belonginboth1}
\alpha t_u+\beta t_v+x(\alpha y_u+\beta y_v)-y(\alpha x_u+\beta x_v)=0 ,
\end{gather}
equivalently,
\begin{gather}\label{belonginboth2}
\alpha(t_u+xy_u-y x_u)+ \beta(t_v+xy_v-y x_v) =0.
\end{gather}
\end{lem}

\begin{proof}First, we assume that $E=\alpha\mathbb{X}_u + \beta\mathbb{X}_v=\alpha(x_u, y_u,z_u)+\beta(x_v, y_v,z_v)=c\mathring{e}_1+d\mathring{e}_2=(c,d,cy-dx)$ for some constants $c$ and $d$. Compare each component of $E$ to have
\begin{gather*}
\alpha x_u+\beta x_v= c, \qquad \alpha y_u+\beta y_v= d, \qquad \alpha z_u+\beta z_v= cy-dx.
\end{gather*}
Substitute the last equation by the f\/irst two, we get the necessary condition $\alpha z_u+\beta z_v= (\alpha x_u+\beta x_v)y - (\alpha y_u+\beta y_v)x$.

The reverse part can be obtained by the direct computation
\begin{gather*}
E=(\alpha x_u+\beta x_v, \alpha y_u+\beta y_v, \alpha z_u+\beta z_v) =(\alpha x_u+\beta x_v)(1,0,y)+(\alpha y_u+\beta y_v)(0,1,-x) \\
\hphantom{E=}{} + \big(0,0,\alpha z_u+\beta z_v -y(\alpha x_u+\beta x_v)+x(\alpha y_u+\beta y_v)\big)\\
\hphantom{E}{}=(\alpha x_u+\beta x_v)\mathring{e}_1+(\alpha y_u+\beta y_v)\mathring{e}_2.
\end{gather*}
We have used the condition \eqref{belonginboth1} in the last equality.
\end{proof}

Next we show a formula for the change of coordinates between the coframe and the coordinates of the surface.
\begin{prop}\label{changofcoordinates2}
Suppose we choose the frames $\{\mathbb{X}(u,v);e_1(\theta),e_2(\theta),T \}$ on $D$ and the coframe with the connection $1$-form defined by \eqref{frames}. We have the identity
\begin{gather}\label{otou}
\Theta\wedge \omega^2\wedge {\omega_1}^2 =\langle E,e_2\rangle {\rm d}u\wedge {\rm d}v \wedge {\rm d}\theta,
\end{gather} where the singular foliation
\begin{gather*}
E:=(z_u+xy_u-yx_u)\mathbb{X}_v-(z_v+xy_v-yx_v)\mathbb{X}_u
\end{gather*} defines the characteristic foliation of $\Sigma$, which is induced from the contact plane $\xi$.
\end{prop}

\begin{proof}
By Proposition~\ref{changofcoordinates1} and the moving frame formula \eqref{movingframe}
\begin{gather*}
\Theta\wedge \omega^2\wedge {\omega_1}^2=({\rm d}z+x{\rm d}y-y{\rm d}x)\wedge\langle {\rm d}\mathbb{X},e_2\rangle \wedge {\rm d}\theta\\
\hphantom{\Theta\wedge \omega^2\wedge {\omega_1}^2}{}
=\big( (z_u+xy_u-yx_u){\rm d}u +(z_v+xy_v-yx_v){\rm d}v\big)\\
\hphantom{\Theta\wedge \omega^2\wedge {\omega_1}^2=}{}
\wedge \big(\langle \mathbb{X}_u, e_2\rangle {\rm d}u\wedge {\rm d}\theta+\langle \mathbb{X}_v,e_2\rangle {\rm d}v\wedge {\rm d}\theta\big) \\
\hphantom{\Theta\wedge \omega^2\wedge {\omega_1}^2}{}
=\langle (z_u+xy_u-yx_u)\mathbb{X}_v-(z_v+xy_v-yx_v)\mathbb{X}_u,e_2\rangle {\rm d}u\wedge {\rm d}v \wedge {\rm d}\theta\\
\hphantom{\Theta\wedge \omega^2\wedge {\omega_1}^2}{}
=\langle E,e_2\rangle {\rm d}u\wedge {\rm d}v \wedge {\rm d}\theta.
\end{gather*}

To prove the vector $E\in TM\cap\xi$, it suf\/f\/ices to show that the coef\/f\/icients $\alpha:=(z_u+xy_u-yx_u)$ and $\beta:= -(z_v+xy_v-yx_v)$ satisfy the condition \eqref{belonginboth2}, and we complete the proof by the previous Lemma~\ref{lemmaiff}.
\end{proof}

\begin{Remark}\label{orientationrem}
In classical integral geometry \cite{Ch, San}, the quantity ${\rm d}L:= {\rm d}p\wedge {\rm d}\theta\wedge {\rm d}t$ is called the \textit{$($kinematic$)$ density} of the line $(p,\theta,t)\in\mathbb{R}^3$, which is always chosen to be positive depending the orientation. Hence, according to~\eqref{otop} and~\eqref{otou}, in the following proof we have to consider the orientation of~$\{(u,v,\theta)\}$ to ensure the positivity of the quantity $\langle E,e_2\rangle $.
\end{Remark}

\begin{proof}[Proof of Theorem~\ref{crofton}]
By Remark~\ref{orientationrem}, we choose ${\rm d}u\wedge {\rm d}v\wedge {\rm d}\theta$ as the orientation of $D$. Let $D=D^+\cup D^-$, where
\begin{gather*}
D^+:=\{ (p,\theta,t,u,v)\,|\,\langle E,e_2\rangle \geq 0\},\qquad
D^-:=\{ (p,\theta,t,u,v)\,|\,\langle E,e_2\rangle \leq 0\},\\
\Gamma:=D^+\cap D^-.
\end{gather*}
By the structure equation \eqref{mcsteq},
\begin{gather}
{\rm d}\big(\Theta\wedge \omega^1\big)={\rm d}\Theta \wedge \omega^1-\Theta\wedge {\rm d}\omega^1 \nonumber\\
\hphantom{{\rm d}\big(\Theta\wedge \omega^1\big)}{} =\big(2\omega^1\wedge\omega^2\big)\wedge\omega^1-\Theta\wedge \big({\omega_1}^2\wedge \omega^2\big) =\Theta\wedge \omega^2\wedge {\omega_1}^2.\label{d2to3}
\end{gather}
We also have
\begin{gather} \Theta\wedge \omega^1 = ({\rm d}z+x{\rm d}y-y{\rm d}x)\wedge\langle {\rm d}\mathbb{X},e_1\rangle \nonumber \\
\hphantom{\Theta\wedge \omega^1}{} =\big( (z_u+xy_u-yx_u){\rm d}u+(z_v+xy_v-yx_v){\rm d}v\big)\wedge \big( \langle \mathbb{X}_u,e_1\rangle {\rm d}u+\langle \mathbb{X}_v,e_1\rangle {\rm d}v\big) \nonumber\\
\hphantom{\Theta\wedge \omega^1}{}=\langle E,e_1\rangle {\rm d}u\wedge {\rm d}v.\label{thetawtouv}
\end{gather}
Now we integrate the kinematic density ${\rm d}L$ over the set $\mathcal{L}$. By using \eqref{otop}, \eqref{d2to3}, the Stock's theorem, and \eqref{thetawtouv}, we have
\begin{gather}
\int_{\ell\in\mathcal{L},\, \ell\cap\Sigma\neq\varnothing}n(\ell\cap \Sigma){\rm d}L =2\bigg(\int_{D^+}\pi^*{\rm d}L-\int_{D^-}\pi^*{\rm d}L\bigg) \nonumber\\
\hphantom{\int_{\ell\in\mathcal{L},\, \ell\cap\Sigma\neq\varnothing}n(\ell\cap \Sigma){\rm d}L}{}
=2\bigg(\int_{D^+} \omega^2\wedge \omega^2_1 \wedge \Theta-\int_{D^-} \omega^2\wedge \omega^2_1\wedge \Theta \bigg) \nonumber\\
\hphantom{\int_{\ell\in\mathcal{L},\, \ell\cap\Sigma\neq\varnothing}n(\ell\cap \Sigma){\rm d}L}{} =2\bigg(\int_{\partial D^+}\Theta\wedge \omega^1-\int_{\partial D^-}\Theta\wedge \omega^1\bigg)\nonumber\\
\hphantom{\int_{\ell\in\mathcal{L},\, \ell\cap\Sigma\neq\varnothing}n(\ell\cap \Sigma){\rm d}L}{}=2\bigg(\int_{\Gamma^+\cup\Gamma}\Theta\wedge \omega^1 -\int_{\Gamma^-\cup\Gamma}\Theta\wedge \omega^1\bigg),\label{allintegrals}
\end{gather}
where $\Gamma^{\pm}:=\partial D^{\pm}\setminus\Gamma.$ We also point out that the number, $2$, occurs in the f\/irst identity is due to the orientations for each horizontal line.

Next, we show that ${\rm d}u\wedge {\rm d}v=0$ on $\Gamma^{\pm}$. Indeed, by using the coordinates $\{(u,v,\theta)\}$ for the conf\/iguration space $D$, any vector f\/ield def\/ined on $\Gamma^+$ can be represented by $A\wedge\frac{\partial}{\partial \theta}\in \partial \Sigma\times S^1$ for some vector $A$ def\/ined on the tangent bundle $T\partial\Sigma$. The value ${\rm d}u\wedge {\rm d}v$ evaluated on $\Gamma^+$ must be
\begin{gather*}
{\rm d}u\wedge {\rm d}v\left(A\wedge \frac{\partial}{\partial\theta}\right)={\rm d}u(A){\cancel{{\rm d}v\left(\frac{\partial}{\partial\theta}\right)}}^{\ =0}-{\rm d}v(A){\cancel{{\rm d}u\left(\frac{\partial}{\partial\theta}\right)}}^{\ =0}=0.
\end{gather*}
Therefore, \eqref{allintegrals} becomes
\begin{gather*}
\int_{\ell\in\mathcal{L},\, \ell\cap\Sigma\neq\varnothing}n(\ell\cap \Sigma){\rm d}L =2\bigg( 2\int_{\Gamma}\Theta\wedge \omega^1 +\int_{\Gamma^+}\Theta\wedge \omega^1-\int_{\Gamma^-}\Theta\wedge \omega^1 \bigg)\\
\hphantom{\int_{\ell\in\mathcal{L},\, \ell\cap\Sigma\neq\varnothing}n(\ell\cap \Sigma){\rm d}L}{} =4\int_\Gamma\Theta\wedge \omega^1 =4\int_{\Gamma}|E|{\rm d}u\wedge {\rm d}v=4\cdot \pa(\Sigma),
\end{gather*} we have used \eqref{thetawtouv} and $E$ is parallel to $e_1$ on $\Gamma$ at the third equality.
\end{proof}

\subsection*{Acknowledgements}
The f\/irst and second authors' research was supported by NCTS grant NSC-100-2628-M-008-001-MY4. They would like to express their appreciation to Professors Jih-Hsin Cheng and Paul Yang for their interests in this work and inspiring discussions. The third author would like to express her thanks to Professor Shu-Cheng Chang for his teaching, constant encouragement, and support. We all thank the anonymous referees for their careful reading of our manuscript and their many insightful comments and suggestions to improve the paper.

\pdfbookmark[1]{References}{ref}
\LastPageEnding


\begin{thebibliography}{99}
\footnotesize\itemsep=0pt

\bibitem{CC}
Calin O., Chang D.-C., Sub-{R}iemannian geometry. General theory and examples,
 \href{https://doi.org/10.1017/CBO9781139195966}{\textit{Encyclopedia of Mathematics and its Applications}}, Vol.~126,
 Cambridge University Press, Cambridge, 2009.

\bibitem{CCG}
Calin O., Chang D.-C., Greiner P., Geometric analysis on the {H}eisenberg group
 and its generalizations, \href{https://doi.org/10.1090/amsip/040}{\textit{AMS/IP Studies in Advanced Mathematics}},
 Vol.~40, Amer. Math. Soc., Providence, RI, International Press, Somerville,
 MA, 2007.

\bibitem{CHMY2}
Cheng J.-H., Hwang J.-F., Malchiodi A., Yang P., Minimal surfaces in
 pseudohermitian geometry, \textit{Ann. Sc. Norm. Super. Pisa Cl. Sci.~(5)}
 \textbf{4} (2005), 129--177, \href{https://arxiv.org/abs/math.DG/0401136}{math.DG/0401136}.

\bibitem{CHMY1}
Cheng J.-H., Hwang J.-F., Malchiodi A., Yang P., A {C}odazzi-like equation and
 the singular set for~{$C^1$} smooth surfaces in the {H}eisenberg group,
 \href{https://doi.org/10.1515/CRELLE.2011.159}{\textit{J.~Reine Angew. Math.}} \textbf{671} (2012), 131--198,
 \href{https://arxiv.org/abs/1006.4455}{arXiv:1006.4455}.

\bibitem{Ch}
Chern S.S., Lectures on integral geometry, {A}cademia Sinica, National Taiwan
 University and National Tsinghua University, 1965.

\bibitem{CCL}
Chern S.S., Chen W.H., Lam K.S., Lectures on dif\/ferential geometry,
 \href{https://doi.org/10.1142/3812}{\textit{Series on University Mathematics}}, Vol.~1, World Sci. Publ. Co.,
 Inc., River Edge, NJ, 1999.

\bibitem{C}
Chevalley C., Theory of {L}ie groups.~{I}, \textit{Princeton Mathematical
 Series}, Vol.~8, Princeton University Press, Princeton, NJ, 1999.

\bibitem{CFH}
Chiu H.-L., Feng X., Huang Y.-C., The fundamental theorem of curves and
 classif\/ications in the Heisenberg groups, \href{https://doi.org/10.1016/j.difgeo.2017.11.009}{\textit{Differential Geom. Appl.}}
 \textbf{56} (2018), 161--172, \href{https://arxiv.org/abs/1511.05237}{arXiv:1511.05237}.

\bibitem{CL}
Chiu H.-L., Lai S.-H., The fundamental theorem for hypersurfaces in {H}eisenberg
 groups, \href{https://doi.org/10.1007/s00526-015-0818-1}{\textit{Calc. Var. Partial Differential Equations}} \textbf{54}
 (2015), 1091--1118, \href{https://arxiv.org/abs/1301.6463}{arXiv:1301.6463}.

\bibitem{FT}
Fuchs D., Tabachnikov S., Invariants of {L}egendrian and transverse knots in
 the standard contact space, \href{https://doi.org/10.1016/S0040-9383(96)00035-3}{\textit{Topology}} \textbf{36} (1997), 1025--1053.

\bibitem{Gi}
Geiges H., An introduction to contact topology, \href{https://doi.org/10.1017/CBO9780511611438}{\textit{Cambridge Studies in
 Advanced Mathematics}}, Vol.~109, Cambridge University Press, Cambridge, 2008.

\bibitem{G}
Grif\/f\/iths P., On {C}artan's method of {L}ie groups and moving frames as applied
 to uniqueness and existence questions in dif\/ferential geometry, \href{http://projecteuclid.org/euclid.dmj/1077310738}{\textit{Duke
 Math.~J.}} \textbf{41} (1974), 775--814.

\bibitem{H}
Huang Y.-C., Applications of integral geometry to geometric properties of sets
 in the 3{D}-{H}eisenberg group, \href{https://doi.org/10.1515/agms-2016-0020}{\textit{Anal. Geom. Metr. Spaces}} \textbf{4}
 (2016), 425--435.

\bibitem{IL}
Ivey T.A., Landsberg J.M., Cartan for beginners: dif\/ferential geometry via
 moving frames and exterior dif\/ferential systems, \textit{Graduate Studies in
 Mathematics}, Vol.~61, Amer. Math. Soc., Providence, RI, 2003.

\bibitem{Le1}
Lee J.M., The {F}ef\/ferman metric and pseudo-{H}ermitian invariants,
 \href{https://doi.org/10.2307/2000582}{\textit{Trans. Amer. Math. Soc.}} \textbf{296} (1986), 411--429.

\bibitem{Le2}
Lee J.M., Pseudo-{E}instein structures on {CR} manifolds, \href{https://doi.org/10.2307/2374543}{\textit{Amer.~J.
 Math.}} \textbf{110} (1988), 157--178.

\bibitem{MM}
Maalaoui A., Martino V., The topology of a subspace of the {L}egendrian curves
 on a closed contact 3-manifold, \href{https://doi.org/10.1515/ans-2014-0210}{\textit{Adv. Nonlinear Stud.}} \textbf{14}
 (2014), 393--426, \href{https://arxiv.org/abs/1303.5017}{arXiv:1303.5017}.

\bibitem{Mon}
Montefalcone F., Some relations among volume, intrinsic perimeter and
 one-dimensional restrictions of {BV} functions in {C}arnot groups,
 \textit{Ann. Sc. Norm. Super. Pisa Cl. Sci.~(5)} \textbf{4} (2005), 79--128.

\bibitem{Pansu}
Pansu P., Une in\'egalit\'e isop\'erim\'etrique sur le groupe de {H}eisenberg,
 \textit{C.~R. Acad. Sci. Paris S\'er.~I Math.} \textbf{295} (1982), 127--130.

\bibitem{PRS}
Prandi D., Rizzi L., Seri M., A sub-Riemannian Santal\'o formula with
 applications to isoperimetric inequa\-li\-ties and f\/irst Dirichlet eigenvalue of
 hypoelliptic operators, \href{https://arxiv.org/abs/1509.05415}{arXiv:1509.05415}.

\bibitem{Ren}
Ren D.L., Topics in integral geometry, \textit{Series in Pure Mathematics},
 Vol.~19, World Sci. Publ. Co., Inc., River Edge, NJ, 1994.

\bibitem{San}
Santal\'o L.A., Integral geometry and geometric probability, 2nd ed., \href{https://doi.org/10.1017/CBO9780511617331}{\textit{Cambridge
 Mathematical Library}}, Cambridge University Press, Cambridge, 2004.

\bibitem{S}
Sharpe R.W., Dif\/ferential geometry. Cartan's generalization of Klein's Erlangen
 program, \textit{Graduate Texts in Mathematics}, Vol.~166, Springer-Verlag,
 New York, 1997.

\bibitem{We}
Webster S.M., Pseudo-{H}ermitian structures on a real hypersurface,
 \href{https://doi.org/10.4310/jdg/1214434345}{\textit{J.~Differential Geom.}} \textbf{13} (1978), 25--41.

\end{thebibliography}
\end{document}